\documentclass[12pt, reqno, 
]{amsart}
\usepackage{amsmath}
\usepackage{amstext}
\usepackage{amsbsy}
\usepackage{amsopn}
\usepackage{upref}
\usepackage{amsthm}
\usepackage{amsfonts}
\usepackage{amssymb}
\usepackage{mathrsfs}
\usepackage{times}
\usepackage{fancybox} 
\usepackage{ascmac}
\allowdisplaybreaks
 \newtheorem{theorem}{Theorem}[section]
 \newtheorem{corollary}[theorem]{Corollary}
 \newtheorem{proposition}[theorem]{Proposition}
 
 \newtheorem{example}[theorem]{Example}
\theoremstyle{definition}
 
 \theoremstyle{remark} 
 \newtheorem{remark}[theorem]{Remark}
\newcommand{\ep}{\varepsilon}
\newcommand{\p}{\partial}
\newcommand{\RR}{\mathbb{R}}
\newcommand{\TT}{\mathbb{T}}
\renewcommand{\Re}{\operatorname{Re}}
\renewcommand{\Im}{\operatorname{Im}}
\DeclareMathOperator{\diag}{diag}
\numberwithin{equation}{section}
\begin{document}
\title[]
{Local well-posedness for a fourth-order nonlinear dispersive system 
on the 1D torus
}
\author[E.~Onodera]{Eiji Onodera}
\address[Eiji Onodera]{Department of Mathematics and Physics, 
Faculty of Science and Technology, 
Kochi University, 
Kochi 780-8520, 
Japan}
\email{onodera@kochi-u.ac.jp}
%
%
%
\maketitle
\begin{abstract}
This paper is concerned with the initial value problem for 
a system of one-dimensional fourth-order 
dispersive partial differential equations on the torus 
with nonlinearity involving derivatives up to second order. 
This paper gives sufficient conditions on the coefficients of the system 
for the initial value problem to be time-locally well-posed 
in Sobolev spaces with high regularity. 
The proof is based on the energy method  
combined with the idea of a gauge transformation 
and the technique of Bona-Smith type parabolic regularization.  
The sufficient conditions can been found 
in connection with  
geometric analysis on a fourth-order geometric 
dispersive partial differential equation for curve flows 
on a compact locally Hermitian symmetric space. 
\end{abstract}
\section{Introduction}
\label{section:introduction}
This paper investigates the initial value problem (IVP) for an $n$-component 
system of fourth-order nonlinear  
dispersive partial differential equations (PDEs): 
 \begin{alignat}{2}
 \left(
 \p_t-iM_a\p_x^4-iM_{\lambda}\p_x^2
 \right)
 Q
  &=F(Q, \p_xQ, \p_x^2Q)
  & \quad
  & \text{in}
    \quad
    \RR\times\TT,
 \label{eq:apde}
 \\
  Q(0,x)
   &=
   Q_{0}(x)
 & \quad
 & \text{in}
   \quad
   \TT. 
 \label{eq:adata}
 \end{alignat}
Here, $n$ is a positive integer, 
$Q={}^t(Q_1,\ldots,Q_n)(t,x):\RR\times \TT\to \mathbb{C}^n$ 
 is an unknown function,  
 $Q_0={}^t(Q_{01},\ldots,Q_{0n})(x):\TT\to \mathbb{C}^n$ 
 is a given initial function, 
$i=\sqrt{-1}$, $\TT=\RR / 2\pi \mathbb{Z}$, 
$M_a=\diag(a_1,\ldots,a_n)$ for $(a_1,\ldots,a_n)\in (\RR\backslash \left\{0\right\})^n$, 
$M_{\lambda}=\diag(\lambda_1,\ldots,\lambda_n)$ for 
$(\lambda_1,\ldots, \lambda_n)\in \RR^n$, and
 \begin{align}
   F(Q,\p_xQ,\p_x^2Q)
   &=
   {}^t(F_1(Q,\p_xQ,\p_x^2Q),\ldots,F_n(Q,\p_xQ,\p_x^2Q)) 
 \nonumber
 \end{align}
where
each $F_j(u,v,w)$ for $j\in \left\{1,\ldots,n\right\}$ is a 
complex-valued polynomial in $u,\overline{u}$, 
$v, \overline{v}$,  
$w,\overline{w}\in \mathbb{C}^{n}$ 
satisfying 
\begin{align}
F_j(Q,\p_xQ, \p_x^2Q)
&=
\sum_{p,q,r=1}^n
\omega_{p,q,r}^{1,j}
\p_x^2Q_p\overline{Q_q}Q_r 
+
\sum_{p,q,r=1}^n
\omega_{p,q,r}^{2,j}
\overline{\p_x^2Q_p}Q_qQ_r
\nonumber
\\
&\quad
+\sum_{p,q,r=1}^n
\omega_{p,q,r}^{3,j}
\p_xQ_p\overline{\p_xQ_q}Q_r 
+
\sum_{p,q,r=1}^n
\omega_{p,q,r}^{4,j}
\p_xQ_p\p_xQ_q\overline{Q_r}
\nonumber 
\\
&\quad 
+O(|Q|^3+|Q|^5), 
\label{eq:b21}
\end{align}
and $\omega_{p,q,r}^{k,j}$ for 
$p,q,r,j\in \left\{1,\ldots, n\right\}$ and 
$k\in \left\{1,\ldots,4\right\}$
are complex constants. 
There is no loss of generality in assuming 
the following conditions (A1) and (A2): 
\begin{enumerate}
\item[(A1)] $\omega_{p,q,r}^{2,j}=\omega_{p,r,q}^{2,j}$ 
for any $p,q,r,j\in \left\{1,\ldots,n\right\}$. 
\item[(A2)] $\omega_{p,q,r}^{4,j}=\omega_{q,p,r}^{4,j}$ 
for any $p,q,r,j\in \left\{1,\ldots,n\right\}$.
\end{enumerate}  
  \par 
  Without the periodic boundary condition on the solution, 
  examples of \eqref{eq:apde} arise in some fields of nonlinear science. 
The following is an example of \eqref{eq:apde} for $n=1$:  
  \begin{align}
  (\p_t-i\nu\p_x^4-i\p_x^2)\psi
  &=
  i\mu_1|\psi|^2\psi
  +i\mu_2|\psi|^4\psi
  +i\mu_3(\p_x\psi)^2\overline{\psi}
  \nonumber
  \\
  &\quad
   +i\mu_4|\p_x\psi|^2\psi
  +i\mu_5\psi^2\overline{\p_x^2\psi}
  +i\mu_6|\psi|^2\p_x^2\psi, 
  \label{eq:4shro}
  \end{align}
where $\psi=\psi(t,x):\RR\times \RR\to \mathbb{C}$ is a solution,  
and $\mu_k$ for $k\in \left\{1,\ldots,6\right\}$ and $\nu\ne 0$ 
  are all real constants. 
The equation is known to arise in relation with 
the continuum limit of a one-dimensional Heisenberg ferromagnetic 
spin chain systems(\cite{DKA,LPD,PDL}), 
the three-dimensional motion of a vortex filament(\cite{fukumoto,FM}), 
and the molecular excitations along the hydrogen bonding spine
 in an alpha-helical protein(\cite{DL}).  
Another example of \eqref{eq:apde} for $n\geqslant 2$ appears in 
\cite{WZY}, which is formulated by 
 \begin{align}
   Q_t
   &=i\alpha \left(
   \frac{1}{2} 
   Q_{xx}+QQ^{*}Q
   \right)
      +i\gamma
      \Biggl[
      \dfrac{1}{2}Q_{xxxx}
      +Q(Q_x^{*}Q)_x
      +Q_xQ_x^{*}Q
  \nonumber
  \\
   &\quad \quad
   +2(
   Q_{xx}Q^{*}Q
   +
   QQ^{*}Q_{xx}
   )
   +3\left\{
   Q_{x}Q^{*}Q_x
   +
   Q(Q^{*}Q)^2
   \right\}
   \Biggr]
   \label{eq:WZY}
   \end{align}
   for $Q(t,x)={}^t(Q_1(t,x),\ldots,Q_n(t,x)):\RR\times \RR\to \mathbb{C}^n$, 
   where $\gamma\ne 0$ and $\alpha$ are real constants, 
   ``$*$" denotes the Hermitian transpose. 
   It is pointed out in \cite{WZY} that   
   \eqref{eq:WZY} investigates the wave propagation 
   of $n$ distinct ultrashort optical fields in a fiber, 
   and models the broadband, ultrashort pulses propagation. 
\par 
This paper is concerned with the well-posedness of 
\eqref{eq:apde}-\eqref{eq:adata} 
for initial data in Sobolev spaces with high regularity.  
The main difficulty is the so-called loss of derivatives, 
occurring affected by the structure of the nonlinearity of \eqref{eq:apde} with the 
first- and the second-derivatives, 
which prevents the classical energy method from working. 
\par 
In the non-periodic case $x\in \RR$, it is well-known that 
the IVP for some kinds of dispersive PDEs  
admits the presence of such bad lower-order terms to some extent, 
which is essentially thanks to the smoothing effect, 
the solutions of which gain extra regularity with respect to the 
initial data. 
Indeed, in the case $n=1$ where \eqref{eq:apde} is a single 
equation, 
there are many literature on the 
well-posedness of \eqref{eq:apde}-\eqref{eq:adata} for initial data in Sobolev spaces 
$H^s(\RR; \mathbb{C})$ for some real number $s>0$. 
See 
\cite{HHW2006,HHW2007,HIT,HJ2005,HJ2007,HJ2011,RWZ,segata2003,segata2004} for more details. 
Also in the general case $n\geqslant 1$,  
the recent study \cite{onodera6}  
has demonstrated that \eqref{eq:apde}-\eqref{eq:adata} is time-locally well-posed  
in $H^m(\RR; \mathbb{C}^n)$ for integers $m\geqslant 4$. 
We mention that the time-local well-posedness can be  
derived without any restrictions on the coefficients 
$\omega_{p,q,r}^{k,j}$. 
\par 
In contrast, our periodic case $x\in \TT$ seems to require a strong 
structure on the nonlinearity of \eqref{eq:apde} for the IVP  
to be time-locally well-posed, because 
the smoothing effect on $\RR$ 
breaks down since the spatial domain $\TT$ is compact.   
The periodic setting has been investigated only for the case $n=1$ 
by Segata in \cite{segata} with an interest in the stability of a periodic 
standing wave solution. 
More concretely, 
applying the so-called modified energy method, 
he showed the IVP for 
\eqref{eq:4shro} on $\TT$ is time-locally well-posed in $H^m(\TT;\mathbb{C})$ 
 for integers $m\geqslant 4$. 
This can be interpreted that 
 the local well-posedness is established  
 under the assumption that all the coefficients $\mu_1,\ldots,\mu_6$ 
 in the nonlinearity are real.   
 He also showed it is time-globally well-posed 
 under an additional assumption on the 
 coefficients $\nu$ and $\mu_1,\ldots,\mu_6$.  
 To the best of the present author's knowledge, 
 no other previous results on the well-posedness  
 for \eqref{eq:apde} are available. 
  The purpose of the present paper is to give sufficient conditions 
  on the coefficients of \eqref{eq:apde} for 
  the IVP with $n\geqslant 1$ to be time-locally well-posedness, 
  as an extension of the previous result  in \cite{segata} for 
 \eqref{eq:4shro} to that for multi-component systems.  
  The motivation comes from a pure mathematical point of view, aiming at finding a 
  good solvable structure of \eqref{eq:apde} as a system which has not been focused 
  in the study of single equations. 
\par 
To state our main results, we introduce the following 
conditions (B1)-(B6): 
\begin{enumerate}
\item[(B1)] $\omega_{p,q,r}^{1,j}=\omega_{r,q,p}^{1,j}$ 
for any $p,q,r,j\in \left\{1,\ldots,n\right\}$. 
\item[(B2)] $\omega_{p,q,r}^{1,j}=-\overline{\omega_{j,r,q}^{1,p}}$ 
for any $p,q,r,j\in \left\{1,\ldots,n\right\}$. 
\item[(B3)] $\omega_{p,q,r}^{2,j}=-\overline{\omega_{r,j,p}^{2,q}}$ 
for any $p,q,r,j\in \left\{1,\ldots,n\right\}$. 
\item[(B4)] $\omega_{p,q,r}^{3,j}=\omega_{r,q,p}^{3,j}$ 
for any $p,q,r,j\in \left\{1,\ldots,n\right\}$. 
\item[(B5)] $\omega_{p,q,r}^{3,j}=-\overline{\omega_{j,r,q}^{3,p}}$ 
for any $p,q,r,j\in \left\{1,\ldots,n\right\}$. 
\item[(B6)] $\omega_{p,q,r}^{4,j}=-\overline{\omega_{j,r,q}^{4,p}}$ 
for any $p,q,r,j\in \left\{1,\ldots,n\right\}$. 
\end{enumerate}  
Our main results is now stated as follows: 
\begin{theorem}
 \label{theorem:lwp} 
 Suppose that $M_a=aI_n$ with $a\neq 0$ where $I_n$ denotes the 
 identity matrix of order $n$ and 
$F(Q,\p_xQ,\p_x^2Q)$ satisfies 
 all the conditions \textup{(A1)-(A2)} and \textup{(B1)-(B6)}. 
 Let $m$ be an integer with $m \geqslant 4$. 
 Then  
 \eqref{eq:apde}-\eqref{eq:adata} 
 is time-locally well-posed in $H^m(\TT;\mathbb{C}^n)$ 
 in the following sense:
 \begin{enumerate}
 \item[\textup{(i)}] (Existence and uniqueness.) 
 For any $Q_0\in H^m(\TT;\mathbb{C}^n)$, there exists 
 $T=T(\|Q_0\|_{H^4})>0$ and a unique solution 
 $Q\in C([-T,T]; H^{m}(\TT;\mathbb{C}^n))$
 to \eqref{eq:apde}-\eqref{eq:adata}.
 \item[\textup{(ii)}] (Continuous dependence with respect to the initial data.) 
 Suppose that $T>0$ and $Q\in C([-T,T];H^m(\TT;\mathbb{C}^n))$ are respectively 
 the time and the unique solution 
 to \eqref{eq:apde} with initial data $Q_0$ obtained in the above part \textup{(i)}. 
 Fix  $T^{\prime}\in (0,T)$. 
 Then for any $\eta>0$, there exists $\delta>0$ such that for any 
 $\widetilde{Q_0}\in H^m(\TT;\mathbb{C}^n)$ satisfying 
 $\|Q_0-\widetilde{Q}_0\|_{H^m}<\delta$, 
 the unique solution $\widetilde{Q}$ 
 to \eqref{eq:apde} with initial data $\widetilde{Q}_0$ exists on 
 $[-T^{\prime},T^{\prime}]\times \TT$ and 
 satisfies 
 $\|Q-\widetilde{Q}\|_{C([-T,T^{\prime}];H^m)}<\eta$.
 \end{enumerate}
 \end{theorem}
\begin{example}
\label{ex:1}
 We consider \eqref{eq:apde} for $n=1$. It is immediate to see that 
 all the conditions \textup{(A1)-(A2)} and \textup{(B1)-(B6)} hold 
 if and only if
$\omega_{1,1,1}^{k,1}\in i\RR$
for all $k\in \left\{1,\ldots,4\right\}$. 
Theorem~\ref{theorem:lwp} 
does not conflict with the fact that the time-local well-posedness for \eqref{eq:4shro} 
on $\TT$ was established in \cite{segata} under the setting 
$i\mu_k\in i\RR$ for all $k\in \left\{1,\ldots,6\right\}$. 
\end{example}
\begin{example}
\label{ex:2}
We consider \eqref{eq:WZY} for $n\geqslant 1$. 
A simple computation  
shows \eqref{eq:WZY} is formulated as 
\eqref{eq:apde}, where 
$M_a=(\gamma/2)I_n$, $M_{\lambda}=(\alpha/2)I_n$, and  
\begin{align}
\omega_{p,q,r}^{1,j}
&=2i\gamma\left(\delta_{pq}\delta_{rj}
+\delta_{pj}\delta_{rq}
\right), 
\quad
\omega_{p,q,r}^{2,j}=
i\dfrac{\gamma}{2}\left(\delta_{pq}\delta_{rj}
+\delta_{pr}\delta_{qj}
\right), 
\nonumber
\\
\omega_{p,q,r}^{3,j}&=
i\gamma\left(\delta_{pq}\delta_{rj}
+\delta_{pj}\delta_{rq}
\right), 
\quad 
\omega_{p,q,r}^{4,j}=
i\dfrac{3\gamma}{2}\left(\delta_{qr}\delta_{pj}
+\delta_{pr}\delta_{qj}
\right)
\nonumber
\end{align}
for all $p,q,r,j\in \left\{1,\ldots,n\right\}$, 
where $\delta$ denotes the Kronecker delta.  
It is easy to check 
all the conditions 
\textup{(A1)-(A2)} 
and 
\textup{(B1)-(B6)} are satisfied. 
\end{example}
\begin{example}
\label{ex:3}
Unlike the above two examples, the sets 
$\left\{\omega_{p,q,r}^{k,j} \mid p,q,r,j\in \left\{1,\ldots,n\right\}
\right\}$ 
for $k\in \left\{1,\ldots,4\right\}$    
under \textup{(A1)-(A2)} 
and 
\textup{(B1)-(B6)}
are not necessarily included in $i\RR$ in the case $n\geqslant 2$. 
We will present such an example of \eqref{eq:apde} for $n=2$ 
in Section~\ref{section:examples}. 
The difference with Example~\ref{ex:1} 
for $n=1$ 
indicates that dispersive systems have a room for well-posedness 
compared with single equations. 
\end{example} 
\par 
Our proof of Theorem~\ref{theorem:lwp} is based on the parabolic regularization of the Bona-Smith type
 and the energy method combined with the idea of  
 a gauge transformation. 
 The gauge transformation we call here behaves 
 as a summation of the identity and a pseudodifferential operator of order $-2$, and the commutator with 
  the fourth-order principal part of \eqref{eq:apde}  
  cancels out the loss of derivatives. 
  The actual proof of the existence of a solution and the continuous dependence with respect to the initial data, 
  corresponding to  Proposition~\ref{proposition:preloc} 
  and \eqref{eq:mcauchy} in Proposition~\ref{proposition:cauchy}, 
  is derived by suitable energy estimates 
  for higher order $x$-derivatives of the solution. 
  Since the gauge transformation acts on images of the 
  partial differentiation $\p_x^2$ in the from 
  \eqref{eq:Nml}-\eqref{eq:V_j} and analogically 
  \eqref{eq:Ekl}-\eqref{eq:Z_j}
  by introducing operators $\Lambda(Q^{\ep})$ and $\Lambda(Q^{\mu})$, 
explicit pseudodifferential calculus is not required. 
However, the strategy does not valid directly in the actual proof of \eqref{eq:1cauchy} in Proposition~\ref{proposition:cauchy} and 
the uniqueness of the solution, 
  since we show them by estimating 
  the difference of two solutions in $H^1(\TT;\mathbb{C}^n)$. 
  The $H^1$-estimates are required to make our proof rigorous under the regularity 
  $m\geqslant 4$.  
  Hence we adjust the above  gauge transformation 
  to work also in the $H^1$-level, introducing 
  \eqref{eq:5313} and analogically \eqref{eq:p5313} 
  by using $\Lambda(Q^{\mu})$ and $\Lambda(Q^1)$.  
  We can mention this part is the mix of the idea of the gauge transformation 
  and the modified energy method used in \cite{Kwon,segata}. 
  See Remark~\ref{remark:ME} 
  for more details.  
 \par 
 We can find the conditions (B1)-(B6)   
  in connection with the study 
  on a fourth-order geometric dispersive PDE for curve flows: 
  Let $N$ be a compact locally Hermitian symmetric space  
  of complex dimension $n\geqslant 1$ 
  with complex structure $J$ and K\"ahler metric $h$, 
  and let $\nabla$ and $R$ denote the Levi-Civita connection associated to $h$ and the 
  Riemann curvature tensor respectively. 
  Consider  
  \begin{alignat}{2}
     & u_t
      =
      a\,J_u\nabla_x^3u_x
      +
      \lambda\, J_u\nabla_xu_x
      +
      b\, R(\nabla_xu_x,u_x)J_uu_x
      +
      c\, R(J_uu_x,u_x)\nabla_xu_x
    \label{eq:pde}
    \end{alignat}
    for $u=u(t,x):\RR\times X\to N$, where 
    $X=\RR$ or $\TT$, 
    and $a\ne 0$, $b$, $c$, $\lambda$ are real constants,  
    $u_t=du\left(\frac{\p}{\p t}\right)$, $u_x=du\left(\frac{\p}{\p x}\right)$, 
    $\nabla_t$ and $\nabla_x$ denote the covariant derivatives 
    along $u$ with respect to $t$ and $x$ respectively, 
    and $J_u$ denotes the complex structure at $u=u(t,x)\in N$. 
    The equation \eqref{eq:pde} is a fourth-order extension of the 
    one-dimensional Schr\"odinger flow 
    equation (see, e.g., \cite{CSU,Koiso1995,Koiso1997,NSVZ,RRS}), 
    and also coincides with  
    the so-called generalized bi-Schr\"odinger flow equation 
    introduced by Ding and Wang in \cite{DW2018}
    if $c=3(a-b)/2$. 
 In \cite{onodera5,onodera6}, the present author  investigated \eqref{eq:pde} in the case $X=\TT$ 
 and showed time-local existence of a 
 unique solution to the IVP for initial data in a kind of 
 geometric Sobolev space. 
 The proof was based on a geometric energy method 
 combined with a gauge transformation, 
 clarifying the structure of the equation satisfied by $\nabla_x^mu_x$ 
 in detail. 
 Considering them, the present author tried to 
 associate the structure of the equation satisfied by $\nabla_x^mu_x$
 with that satisfied by $\p_x^mQ$ 
 (where $Q$ is a solution to \eqref{eq:apde}) 
 via the so-called Generalized Hasimoto transformation. 
The approach successfully leads us to find
 the conditions (C1)-(C4) and (G1)-(G5) (introduced in Section~\ref{section:condition})
  and the suitable gauge transformation which works to show 
  Theorem~\ref{theorem:lwp}.   
 Moreover,  
 $$
 \textup{(B1)-(B6) }\iff \textup{(C1)-(C4)}
 \iff
 \textup{(G1)-(G5)}
 $$ 
 holds under (A1)-(A2). See Propositions~\ref{proposition:BC} and \ref{proposition:GC} 
 in Section~\ref{section:condition}
 for the detail. 
\par 
Additionally, 
the theory on the $L^2$-well-posedness for 
linear dispersive PDEs for complex- or real-valued functions on $\TT$ 
or their two-component systems lies behind the idea  
of the proof of Theorem~\ref{theorem:lwp}:  
In \cite{Mizuhara}, 
Mizuhara considered the following linear dispersive PDE: 
\begin{align}
&\left\{
\p_t-i
\left(
D_x^4+a(x)D_x^3+b(x)D_x^2+c(x)D_x+d(x)
\right)
\right\}
u
=f(t,x)
\label{eq:Mizuhara}
\end{align}
for $u=u(t,x):\RR\times \TT\to \mathbb{C}$, 
where $D_x=-i\p_x$, 
$a(x), b(x), c(x), d(x)\in C^{\infty}(\TT;\mathbb{C})$ 
and $f(t,x)\in L^1_{\text{loc}}(\RR; L^2(\TT; \mathbb{C}))$ are 
given functions. He showed the IVP for \eqref{eq:Mizuhara} is 
$L^2$-well-posed if and only if all the following three conditions hold:
\begin{align}
&\Im
\int_{0}^{2\pi}a(x)dx
=0, 
\quad 
\Im
\int_{0}^{2\pi}\left\{b(x)-\dfrac{3}{8}a(x)^2\right\}dx=0, 
\nonumber
\\
&
\Im
\int_{0}^{2\pi}\left\{
c(x)-\dfrac{a(x)b(x)}{2}+\dfrac{a(x)^3}{8}
\right\}dx=0.
\nonumber
\end{align}
In \cite{chihara2015}, 
Chihara considered the following linear dispersive system: 
\begin{align}
&\left\{
\begin{pmatrix}
   1 & 0 \\
   0 & 1
\end{pmatrix}\p_t+i
\left(
\begin{pmatrix}
   1 & 0 \\
   0 & -1
\end{pmatrix}D_x^4+A(x)D_x^3+B(x)D_x^2+C(x)D_x+D(x)
\right)
\right\}
{\bf u}
\nonumber
\\
&
={\bf f}(t,x)
\label{eq:chihara1}
\end{align}
for ${\bf u}={\bf u}(t,x):\RR\times \TT\to \mathbb{C}^2$, 
where 
$A(x), B(x), C(x), D(x)\in C^{\infty}(\TT;M_2(\mathbb{C}))$ 
and ${\bf f}(t,x)\in L^1_{\text{loc}}(\RR; L^2(\TT; \mathbb{C}^2))$ are 
given functions.   
He gave a necessary and sufficient condition on 
$A(x),$ $B(x)$, $C(x)$, $D(x)$ for the IVP of  \eqref{eq:chihara1} 
to be $L^2$-well-posed. 
If we assume $A(x)\equiv 0$ for simplicity, 
then the necessary and sufficient condition
reads 
\begin{align}
\Im
\int_{0}^{2\pi}b_{kk}(x)dx
=0
\ \ \text{and}
\ \
\Im
\int_{0}^{2\pi}c_{kk}(x)dx=0 
\ \
\text{for $k=1, 2$}, 
\nonumber
\end{align}
where $b_{kk}(x)$ and $c_{kk}(x)$ denote 
the $(k,k)$-component of $B(x)$ and that of $C(x)$ respectively. 
He also considered the system as a specialization of \eqref{eq:chihara1}: 
\begin{align}
\left\{
\begin{pmatrix}
   1 & 0 \\
   0 & 1
\end{pmatrix}\p_t+
\left(
\begin{pmatrix}
   0 & -1 \\
   1 & 0
\end{pmatrix}\p_x^4+\beta(x)\p_x^2+\gamma(x)\p_x
\right)
\right\}
{\bf w}&
={\bf h}(t,x)
\label{eq:chihara2}
\end{align}
for ${\bf w}={\bf w}(t,x):\RR\times \TT\to \mathbb{R}^2$, 
where 
$\beta(x), \gamma(x)\in C^{\infty}(\TT;M_2(\mathbb{R}))$ 
and ${\bf h}(t,x)\in L^1_{\text{loc}}(\RR; L^2(\TT; \mathbb{R}^2))$ are 
given functions. 
He gave a necessary and sufficient condition for the IVP of  \eqref{eq:chihara2} 
to be $L^2$-well-posed, 
presenting another proof of the sufficiency of the condition 
with the application to a geometric dispersive PDE 
for closed curve flows on a compact Riemann surface in mind. 
The proof of the time-local existence and uniqueness results for the IVP of \eqref{eq:pde} in \cite{onodera4,onodera5} is actually inspired by the proof of the sufficiency.   
This motivates us to apply the theory to our problem \eqref{eq:apde}-\eqref{eq:adata} as well. 
Applying it involves a  
classification of the nonlinear terms by distinguishing their coefficient part, 
which seemed to  be a nontrivial task  
because the nonlinearity in \eqref{eq:apde} consists of the cubic products  
of $Q$, $\overline{Q}$, $\p_xQ$, $\overline{\p_xQ}$, 
$\p_x^2Q$, $\overline{\p_x^2Q}$. 
Taking them into account,   
the present author chose to try to   
associate with the study on \eqref{eq:pde} as stated above. 
By doing this, the structure of the equation for $\p_x^mQ$ is obtained clearly 
as \eqref{eq:b384} where 
operators $P_k^{\ell}(Q)$ (defined by \eqref{eq:b511}-\eqref{eq:b515}) 
turn out to play roles to distinguish the coefficient parts 
of nonlinear terms and to posses 
crucial properties (Proposition~\ref{proposition:symskewsym}) 
to prove Theorem~\ref{theorem:lwp}.   
Additionally, the $n$-component case for general $n\geqslant 1$ 
can be handled under the strategy. 
\par 
The assumption  
$M_a=aI_n$ 
ensures 
\begin{equation}
\left[M_a, \Lambda(Q^{\ep})\right]=
\left[M_a, \Lambda(Q^{\mu})\right]=\left[M_a, \Lambda(Q^{1})\right]=0. 
\label{eq:comm930}
\end{equation} 
Without \eqref{eq:comm930}, another loss of derivatives occurs and the argument 
of our proof breaks down. 
In other words, even if $M_a\ne aI_n$, 
our argument of the proof is still valid 
under an exceptional case where 
\eqref{eq:comm930} is satisfied.  
We present the example as follows:  
\begin{enumerate}
\item[(B7)] 
$\omega_{k,q,r}^{1,j}
=0$ 
unless $j=k$ 
for all $q,r,j,k\in \left\{1,\ldots,n\right\}$.  
\item[(B8)] $\omega_{k,q,r}^{2,j}=0$ unless $j=k$ 
for any $q,r,j,k\in \left\{1,\ldots,n\right\}$. 
\item[(B9)] $\omega_{k,q,r}^{3,j}+2\omega_{k,r,q}^{4,j}=0$ 
unless $j=k$ for any $q,r,j,k\in \left\{1,\ldots,n\right\}$. 
\end{enumerate}
\begin{corollary}
\label{cor:cor930} 
Under the same assumptions of Theorem~\ref{theorem:lwp} 
with additional conditions \textup{(B7)-(B9)},  
the IVP 
 \eqref{eq:apde}-\eqref{eq:adata} 
 is time-locally well-posed in $H^m(\TT;\mathbb{C}^n)$.
\end{corollary}
The  additional (B7)-(B9) correspond to be  
sufficient conditions 
that each of $\Lambda(Q^{\ep})$, $\Lambda(Q^{\mu})$, $\Lambda(Q^{1})$  
behaves as a diagonal matrix.  
Unfortunately however, the 
conditions may be a bit artificial and rather strong. See Remark~\ref{remark:add930}.
\par 
The organization of the present paper is as follows: 
Section~\ref{section:condition} gives the background on 
the conditions (B1)-(B6). 
Sections~\ref{section:local}-\ref{section:prooflw} 
complete the proof of Theorem~\ref{theorem:lwp} 
and Corollary~\ref{cor:cor930}. 
Section~\ref{section:examples} is on Example~\ref{ex:3}. 
Section~\ref{section:GB}, which can be read independently, 
explains the background of our results with relevance to the study on 
\eqref{eq:pde}.  
 \section*{Notation used throughout this paper}
\label{section:notation} 
      Different positive constants are sometimes denoted by the same $C$ for simplicity, 
      if there seems to be no confusion. 
      Expressions such as $C=C(\cdot,\ldots,\cdot)$ and  
      $C_k=C_k(\cdot,\ldots,\cdot)$ are also sometimes 
      used to show the dependency on quantities appearing in parenthesis. 
      Other symbols to denote a constant are explained on each occasion.
   \par 
   For any 
   $z={}^{t}(z_1,\ldots,z_n)$ and $w={}^{t}(w_1,\ldots,w_n)$ 
   in $\mathbb{C}^n$, their inner product is defined by 
   $z\cdot w=\sum_{j=1}^{n}z_j\overline{w_j}$, 
   and the norm of $z$ is by 
   $|z|=(z\cdot z)^{1/2}$. 
   \par 
  The $L^2$-space of $\mathbb{C}^n$-valued functions on $\TT$ is denoted by 
  $L^2(\TT;\mathbb{C}^n)$ being the set of all measurable functions 
    $f={}^{t}(f_1,\ldots,f_n):\TT\to \mathbb{C}^n$ such that 
    $$
    \|f\|_{L^2}:=\left(
    \int_{\TT}|f(x)|^2\,dx
    \right)^{1/2}
    =\left(
    \sum_{j=1}^n
      \int_{0}^{2\pi}f_j(x)\overline{f_j(x)}\,dx
      \right)^{1/2}
    <\infty.
    $$
For any $f={}^{t}(f_1,\ldots,f_n)$ and  
$g={}^{t}(g_1,\ldots,g_n)$ in $L^2(\TT;\mathbb{C}^n)$, 
we set 
\begin{align}
\langle
f,g
\rangle
&:=
\int_{\TT} f(x)\cdot g(x) \,dx 
=
\sum_{j=1}^n
\int_{0}^{2\pi}f_j(x)\overline{g_j(x)}\,dx,   
\nonumber
\end{align}
which lets us  write 
$\|f\|_{L^2}^2=
\langle
f,f
\rangle$.  
  The $L^2$-type  Sobolev space for a nonnegative integer $k$ is denoted by 
  $H^k(\TT;\mathbb{C}^n)$  
   being the set of all measurable functions 
   $f={}^{t}(f_1,\ldots,f_n):\TT\to \mathbb{C}^n$ such that 
  $\p_x^{\ell}f\in L^2(\TT;\mathbb{C}^n)$ for  
  all $\ell\in \left\{0,\ldots,k\right\}$. 
  The norm $\|f\|_{H^k}$ of $f\in H^k(\TT;\mathbb{C}^n)$ 
  is defined to satisfy 
  $\|f\|_{H^k}^2=
   \sum_{\ell=0}^{k}\|\p_x^{\ell}f\|_{L^2}^2.
   $ 
   Moreover, $H^{\infty}(\TT;\mathbb{C}^n)$ denotes the intersection 
   of all $H^k(\TT;\mathbb{C}^n)$ for $k=0,1,2,\ldots$, 
and $C([t_1,t_2];H^k(\RR;\mathbb{C}^n))$ denotes the 
Banach space of $H^k(\TT;\mathbb{C}^n)$-valued continuous functions on 
the interval $[t_1,t_2]$ with the norm 
$\|Q\|_{C([t_1,t_2];H^k)}:=\sup_{t\in [t_1,t_2]}\|Q(t,\cdot)\|_{H^k}$.
\section{Remark on the sufficient condition}
\label{section:condition} 
We introduce the set $\Gamma:=\Gamma_1\cap \Gamma_2$, 
where  
\begin{align}
\Gamma_1
&=
\left\{
f(p,q,r,j):\left\{1,\ldots,n\right\}^4\to \mathbb{C}
\middle| 
\begin{array}{l}
\text{
$f(p,q,r,j)=f(r,q,p,j)$  
}
\\
\text{for all $p,q,r,j\in \left\{1,\ldots,n\right\}$}
 \end{array}
\right\}, 
\nonumber
\\
\Gamma_2
&=
\left\{
f(p,q,r,j):\left\{1,\ldots,n\right\}^4\to \mathbb{C}
\middle| 
\begin{array}{l}
\text{
$f(p,q,r,j)=\overline{f(j,r,q,p)}$
}
\\
\text{for all $p,q,r,j\in \left\{1,\ldots,n\right\}$}
 \end{array}
\right\}.
\nonumber
\end{align} 
Note that $\Gamma_1$, $\Gamma_2$, $\Gamma$ 
are vector spaces over $\RR$ 
with respect to the operations of pointwise addition and scalar multiplication.   
\begin{proposition}
\label{proposition:BC}
Under the conditions \textup{(A1)-(A2)}, 
all the conditions 
\textup{(B1)-(B6)}
hold if and only if all the following conditions  
\textup{(C1)-(C4)} hold:
\begin{enumerate}
\item[\textup{(C1)}] $f_1(p,q,r,j):=\omega_{p,q,r}^{1,j}\in i\Gamma$.
\item[\textup{(C2)}] $f_2(p,q,r,j):=\omega_{q,p,r}^{2,j}\in i\Gamma$.
\item[\textup{(C3)}] $f_3(p,q,r,j):=\omega_{p,q,r}^{3,j}\in i\Gamma$.
\item[\textup{(C4)}] $f_4(p,q,r,j):=\omega_{p,r,q}^{4,j}\in i\Gamma$.
\end{enumerate}  
\end{proposition}
This proposition easily follows from the definition of $\Gamma$.  
More concretely, the following four equivalences hold:  
\begin{alignat}{2}
\text{(C1)}&\iff 
(\text{(B1) and (B2)}),  
\quad 
\text{(C2)}&\iff 
(\text{(A1) and (B3)}),  
\nonumber
\\
\text{(C3)}&\iff 
(\text{(B4) and (B5)}),  
\quad 
\text{(C4)}&\iff 
(\text{(A2) and (B6)}). 
\nonumber
\end{alignat}
We omit the detail. 
\par
We next consider the following conditions (G1)-(G5): 
\begin{enumerate}
\item[(G\textup{$\ell$})] 
$g_{\ell}(p,q,r,j):=S_{p,q,r}^{\ell,j}\in \Gamma$
\quad ($\ell\in \left\{1,\ldots,5\right\}$), 
\end{enumerate}
where 
\begin{align}
S_{p,q,r}^{1,j}
&=
i\omega_{q,p,r}^{2,j}, 
\label{eq:S1}
\\
S_{p,q,r}^{2,j}
&=
-\dfrac{i}{2}
\left(
\omega_{r,q,p}^{1,j}-\omega_{q,r,p}^{2,j}
\right), 
\label{eq:S2}
\\
S_{p,q,r}^{3,j}
&=
-2S_{p,q,r}^{2,j}
-\dfrac{i}{2}
\left(
2\omega_{r,q,p}^{1,j}+\omega_{r,q,p}^{3,j}+\omega_{r,p,q}^{4,j}+\omega_{p,r,q}^{4,j}
\right), 
\label{eq:S3}
\\
S_{p,q,r}^{4,j}
&=
-\dfrac{i}{2}
\left(
-\omega_{r,q,p}^{3,j}+\omega_{r,p,q}^{4,j}+\omega_{p,r,q}^{4,j}
\right), 
\label{eq:S4}
\\
S_{p,q,r}^{5,j}
&=
\dfrac{i}{2}
\left(
\omega_{q,p,r}^{2,j}+\omega_{q,r,p}^{2,j}+\omega_{p,q,r}^{3,j}
\right). 
\label{eq:S5}
\end{align}
\begin{proposition}
\label{proposition:GC}
Under the condition \textup{(A2)}, the following conditions are 
equivalent: 
\begin{enumerate}
\item[\textup{(i)}] All the conditions \textup{(C1)-(C4)} hold. 
\item[\textup{(ii)}] All the conditions \textup{(G1)-(G2) and (G4)-(G5)} hold.
\item[\textup{(iii)}] All the conditions \textup{(G1)-(G5)} hold. 
\end{enumerate}
\end{proposition}
\begin{proof}[Proof of Proposition~\ref{proposition:GC}] 
(i) $\Longrightarrow$ (ii): 
We shall check all the conditions (G1)-(G2) and (G4)-(G5) hold. 
Since $\Gamma$ is a vector space over $\RR$,  
it suffices to check each of $g_{\ell}(p,q,r,j)$ for $\ell=1,2,4,5$ 
is expressed as a linear combination 
of $-if_{k}(p,q,r,j)\in \Gamma$ for 
$k\in \left\{1,\ldots,4\right\}$ with real coefficients. 
First, 
it is easy to see 
$g_1(p,q,r,j)=S_{p,q,r}^{1,j}=i\omega_{q,p,r}^{2,j}=if_2(p,q,r,j)\in \Gamma$, which follows from \eqref{eq:S1} and (C2). 
Second, 
by the definition of $f_1(p,q,r,j)$ and $f_2(p,q,r,j)$ and \eqref{eq:S2},   
\begin{align}
g_2(p,q,r,j)
&=
-\dfrac{i}{2}
\left(
\omega_{r,q,p}^{1,j}-\omega_{q,r,p}^{2,j}
\right)
=
-\dfrac{i}{2}
\left\{
f_1(r,q,p,j)
-f_2(r,q,p,j)
\right\}. 
\nonumber
\end{align}
Since $-if_{k}(p,q,r,j)\in \Gamma\subset \Gamma_1$ for 
$k\in \left\{1,\ldots,4\right\}$, 
it follows that 
\begin{equation}
-if_{k}(p,q,r,j)
=-if_k(r,q,p,j)
\quad 
\text{for $k\in \left\{1,\ldots,4\right\}$.}
\label{eq:ifk}
\end{equation}
Using \eqref{eq:ifk} for $k=1,2$ and (C1) and (C2), we see 
\begin{align}
g_2(p,q,r,j)
&=
\dfrac{1}{2}
\left\{
-if_1(p,q,r,j)
\right\}
-\dfrac{1}{2}
\left\{
-i
f_2(p,q,r,j)
\right\}\in \Gamma. 
\nonumber
\end{align}
Third, in the same way above, 
using \eqref{eq:ifk} for $k=3,4$ and (C3) and (C4), 
we see 
\begin{align}
g_4(p,q,r,j)
&=
-\dfrac{i}{2}
\left\{
-f_3(r,q,p,j)+f_4(r,q,p,j)+f_4(p,q,r,j)
\right\}
\nonumber
\\
&=
-\dfrac{i}{2}
\left\{
-f_3(p,q,r,j)+2f_4(p,q,r,j)
\right\}
\in \Gamma. 
\nonumber
\end{align}
Finally, using \eqref{eq:ifk} for $k=2$ and (C2) and (C3), we have
\begin{align}
g_5(p,q,r,j)
&=
\dfrac{i}{2}
\left\{
2f_2(p,q,r,j)+f_3(p,q,r,j)
\right\}
\in \Gamma.
\nonumber
\end{align}
(ii) $\Longrightarrow$ (iii): 
It suffices to show (G3) follows from (G1)-(G2) and (G4)-(G5). 
By (G1), 
\begin{equation}
\label{eq:b1550}
\omega_{q,p,r}^{2,j}=\omega_{q,r,p}^{2,j} 
\quad \text{for any $p,q,r,j\in \left\{1,\ldots,n\right\}$}.
\end{equation} 
Using \eqref{eq:b1550} and rewriting 
\eqref{eq:S1}-\eqref{eq:S5} yields 
\begin{align}
i\omega_{r,q,p}^{1,j}
&=
S_{p,q,r}^{1,j}-2S_{p,q,r}^{2,j},  
\label{eq:b151}
\\
i\omega_{q,p,r}^{2,j}
&=
S_{p,q,r}^{1,j},
\label{eq:b152}
\\
 i(2\omega_{q,p,r}^{2,j}+\omega_{p,q,r}^{3,j})
 &=2S_{p,q,r}^{5,j}, 
 \label{eq:b153}
 \\
i(\omega_{r,q,p}^{1,j} 
+\omega_{r,q,p}^{3,j}) 
&=-2
S_{p,q,r}^{2,j}-
S_{p,q,r}^{3,j}+S_{p,q,r}^{4, j}, 
\label{eq:b154}
\\
i(\omega_{r,q,p}^{1,j} 
+
\omega_{r,p,q}^{4,j}
+
 \omega_{p,r,q}^{4,j})
&=
-2S_{p,q,r}^{2,j}-S_{p,q,r}^{3, j}-S_{p,q,r}^{4,j}.
\label{eq:b155}
\end{align}
Indeed,  \eqref{eq:b151}-\eqref{eq:b153} follow from 
\eqref{eq:S1}-\eqref{eq:S2}, 
\eqref{eq:S5} and \eqref{eq:b1550}.  
Moreover, \eqref{eq:b154} is obtained by subtracting \eqref{eq:S4} from \eqref{eq:S3} , 
and \eqref{eq:b155} is by 
adding \eqref{eq:S3} and \eqref{eq:S4}. 
Furthermore, by subtracting \eqref{eq:b151} from \eqref{eq:b154},  
\begin{align}
i\omega_{r,q,p}^{3,j} 
&=
-S_{p,q,r}^{1,j}
-S_{p,q,r}^{3,j}+S_{p,q,r}^{4, j}. 
\label{eq:ti1}
\end{align}
On the other hand, by \eqref{eq:b152} and \eqref{eq:b153}, we see 
$i\omega_{p,q,r}^{3,j}
 =
 -2S_{p,q,r}^{1,j}+2S_{p,q,r}^{5,j}$.
By (G1) and (G5), the right hand side of the above 
is invariant under the replacement of $p$ and $r$ with each other, 
and so is the left hand side, which implies
\begin{align}
i\omega_{r,q,p}^{3,j}
 &=
 -2S_{p,q,r}^{1,j}+2S_{p,q,r}^{5,j}.
 \label{eq:ti2}
\end{align}
Comparing \eqref{eq:ti1} and \eqref{eq:ti2}, we have 
$S_{p,q,r}^{3,j}
=
S_{p,q,r}^{1,j}+S_{p,q,r}^{4,j}-2S_{p,q,r}^{5,j}$. 
This shows (G3) follows from (G1) and (G4)-(G5). 
\\
(iii) $\Longrightarrow$ (i):  
We shall check all the conditions (C1)-(C4) hold, that is, 
$if_k(p,q,r,j)\in \Gamma$ for each 
$k\in \left\{1,\ldots,4\right\}$.  
First, by (G1), 
\begin{equation}
if_2(p,q,r,j)=i\omega_{q,p,r}^{2,j}=S_{p,q,r}^{1,j}=g_1(p,q,r,j)\in \Gamma. 
\label{eq:9184}
\end{equation}
This shows (C2) holds. 
Second, by the definition of $f_k(p,q,r,j)$ ($k=1,2$) 
and \eqref{eq:S2}, 
\begin{align}
-\dfrac{i}{2}f_1(p,q,r,j)
&=
-\dfrac{i}{2}\omega_{p,q,r}^{1,j}
=S_{r,q,p}^{2,j}
-\dfrac{i}{2}\omega_{q,p,r}^{2,j}
=g_2(r,q,p,j)
-\dfrac{i}{2}f_2(p,q,r,j). 
\nonumber
\end{align}
Noting $g_2(r,q,p,j)=g_2(p,q,r,j)\in \Gamma$ follows from (G2), 
and using \eqref{eq:9184}, we see 
\begin{align}
-\dfrac{i}{2}f_1(p,q,r,j)
&=
g_2(p,q,r,j)
-\dfrac{1}{2}g_1(p,q,r,j)
\in \Gamma.
\nonumber
\end{align}
This shows (C1) holds. 
Third, by \eqref{eq:S5}, 
\begin{align}
\dfrac{i}{2}f_3(p,q,r,j)
&=
\dfrac{i}{2}\omega_{p,q,r}^{3,j}
=
g_5(p,q,r,j)-\dfrac{i}{2}
(
\omega_{q,p,r}^{2,j}+\omega_{q,r,p}^{2,j}
).
\nonumber
\end{align} 
Here, 
$i\omega_{q,p,r}^{2,j}=g_1(p,q,r,j)\in \Gamma$ follows from \eqref{eq:9184}, 
and thus 
$i\omega_{q,p,r}^{2,j}=i\omega_{q,r,p}^{2,j}$.  
Using this, (G1), and (G5), we obtain
\begin{align}
\dfrac{i}{2}f_3(p,q,r,j)
&=
g_5(p,q,r,j)-i
\omega_{q,p,r}^{2,j}
=g_5(p,q,r,j)-g_1(p,q,r,j)\in \Gamma,  
\nonumber
\end{align}  
which shows (C3) holds. 
Finally, adding 
\eqref{eq:S3} and \eqref{eq:S4} yields
\begin{align}
if_4(r,q,p,j)+if_4(p,q,r,j)
&=
i(\omega_{r,p,q}^{4,j}+\omega_{p,r,q}^{4,j})
=
-2S_{p,q,r}^{2,j}-S_{p,q,r}^{3, j}-S_{p,q,r}^{4,j}
-i\omega_{r,q,p}^{1,j}.   
\nonumber
\end{align}
Combining this and \eqref{eq:b151} which follows from 
\eqref{eq:S1}-\eqref{eq:S2} and (G1), 
we obtain 
\begin{align}
if_4(r,q,p,j)+if_4(p,q,r,j)
&=
-g_1(p,q,r,j)-g_3(p,q,r,j)-g_4(p,q,r,j). 
\nonumber 
\end{align}
This shows (C4) holds, since  $f_4(r,q,p,j)=f_4(p,q,r,j)$ is ensured by (A2) 
and since  $g_k(p,q,r,j)\in \Gamma$ for $k=1,3,4$  
follow from 
(G1) and (G3)-(G4). 
\end{proof} 
\par 
Let $m\geqslant 4$ be an integer, and let 
$Q\in C([t_1,t_2]; H^m(\TT;\mathbb{C}^n))$ for some $t_1<t_2$.  
Set $Y={}^t(Y_1,\ldots,Y_n)=\p_xQ$.
A simple computation shows  
\begin{align}
\p_x(F_j(Q,\p_xQ, \p_x^2Q))
&=
\sum_{p,q,r=1}^{n}
\left(
\omega_{p,q,r}^{1,j} \p_x^{2}Y_p\overline{Q_q}Q_r
+
\omega_{p,q,r}^{2,j}\overline{\p_x^{2}Y_p}Q_qQ_r
\right)
\label{eq:nom1}
\\
&\quad 
+
\sum_{p,q,r=1}^{n}
\left(
\omega_{p,q,r}^{1,j} 
+\omega_{p,q,r}^{3,j} 
\right)
\p_xY_p\overline{\p_xQ_q}Q_r
\label{eq:nom2}
\\
&\quad 
+
\sum_{p,q,r=1}^{n}
\left(
\omega_{p,q,r}^{1,j} 
+
\omega_{p,r,q}^{4,j}
+
 \omega_{r,p,q}^{4,j}
\right)
\p_xY_p\overline{Q_q}\p_xQ_r
\label{eq:nom3}
\\
&\quad 
+
\sum_{p,q,r=1}^{n}
\left(
\omega_{p,q,r}^{2,j}+\omega_{p,r,q}^{2,j}
+\omega_{q,p,r}^{3,j} 
\right)
\overline{\p_xY_p}\p_xQ_qQ_r
\label{eq:nom4}
\\&\quad
+
O\left(|\p_xQ|^3+|\p_xQ||Q|^2+|\p_xQ||Q|^4\right). 
\nonumber
\end{align}
Here, we introduce 
$P_k^{\ell}(Q)$ for $k,\ell\in \left\{1,\ldots,5\right\}$ 
acting on the space of $\mathbb{C}^n$-valued functions 
on $[t_1,t_2]\times \TT$, 
which are defined by
$
P_k^{\ell}(Q)v={}^t((P_k^{\ell}(Q)v)_1,\ldots, (P_k^{\ell}(Q)v)_n)
$
for any 
$v={}^t(v_1,\ldots,v_n)(t,x): [t_1,t_2]\times \TT\to \mathbb{C}^n$, where 
\begin{align}
(P_{1}^{\ell}(Q)v)_j
&= 
-i\sum_{p,q,r=1}^n
	  	  S_{p,q,r}^{\ell,j}\left(
	  	  v_p
	  	  \overline{Q_q}
	  	 +Q_p\overline{v_q}
	  	  \right)
	  	  Q_r,	  	  	  
\label{eq:b511}
\\
(P_{2}^{\ell}(Q)v)_j
&=
2i\sum_{p,q,r=1}^n
	  S_{p,q,r}^{\ell,j}
	  Q_p
	  \overline{Q_q}
	  v_r, 
\label{eq:b512}
\\
(P_{3}^{\ell}(Q)v)_j
&=
i\sum_{p,q,r=1}^n
	  	  S_{p,q,r}^{\ell, j}\p_x(Q_p\overline{Q_q})
	  	  v_r, 
\label{eq:b513}
\\
(P_{4}^{\ell}(Q)v)_j
&=
i\sum_{p,q,r=1}^n
	  	  S_{p,q,r}^{\ell, j}\left(
	  	  \p_xQ_p\overline{Q_q}
	  	 -Q_p\overline{\p_xQ_q}
	  	  \right)
	  	  v_r, 
\label{eq:b514}
\\
(P_{5}^{\ell}(Q)v)_j
&=
-2i\sum_{p,q,r=1}^n
	  	  	  	  S_{p,q,r}^{\ell,j}
	  	  	  	  \overline{v_q}
	  	  	  	  \p_xQ_pQ_r.
\label{eq:b515}
\end{align}
\begin{proposition}
\label{proposition:pF}
Set $Y=\p_xQ$ as above. Then 
\begin{align}
\p_x(F(Q, \p_xQ, \p_x^2Q))
&
= 
P_1^1(Q)\p_x^2Y
+
\p_x\left\{
P_2^2(Q)\p_xY
\right\}
+\sum_{k=3}^{5}
P_k^k(Q)\p_xY
\nonumber
\\
&\quad
+O
\left(
|\p_xQ|^3+|\p_xQ||Q|^2+|\p_xQ||Q|^4
\right).
\label{eq:b51new}
\end{align}
\end{proposition} 
\begin{proof}[Proof of Proposition~\ref{proposition:pF}] 
First, 
\begin{align}
\eqref{eq:nom1}
&=
\sum_{p,q,r=1}^{n}
\omega_{p,q,r}^{1,j} \p_x^{2}Y_p\overline{Q_q}Q_r
+
\sum_{p,q,r=1}^{n}
\omega_{q,p,r}^{2,j}Q_p\overline{\p_x^{2}Y_q}Q_r
\nonumber 
\\
&=
\sum_{p,q,r=1}^{n}
\omega_{q,p,r}^{2,j}
\left(
\p_x^{2}Y_p\overline{Q_q}+Q_p\overline{\p_x^{2}Y_q}
\right)Q_r
+
\sum_{p,q,r=1}^{n}
\left(\omega_{p,q,r}^{1,j} -\omega_{q,p,r}^{2,j}\right)
\p_x^{2}Y_p\overline{Q_q}Q_r
\nonumber 
\\
&=
\sum_{p,q,r=1}^{n}
\omega_{q,p,r}^{2,j}
\left(
\p_x^{2}Y_p\overline{Q_q}+Q_p\overline{\p_x^{2}Y_q}
\right)Q_r
+
\sum_{p,q,r=1}^{n}
\left(\omega_{r,q,p}^{1,j} -\omega_{q,r,p}^{2,j}\right)
Q_p\overline{Q_q}\p_x^{2}Y_r. 
\nonumber 
\end{align}
Here, the first (resp. third) equality holds by replacing $p$ and $q$ 
(resp. $p$ and $r$)
in the second summation term of the right hand side.   
Further, by \eqref{eq:S1} and \eqref{eq:S2},   
\begin{align}
\eqref{eq:nom1}
&=
-i
\sum_{p,q,r=1}^{n}
S_{p,q,r}^{1,j}
\left(
\p_x^{2}Y_p\overline{Q_q}+Q_p\overline{\p_x^{2}Y_q}
\right)Q_r
+2i 
\sum_{p,q,r=1}^{n}
S_{p,q,r}^{2,j}
Q_p\overline{Q_q}\p_x^{2}Y_r 
\nonumber
\\
&=
(P_1^1(Q)\p_x^2Y)_j
+(\p_x\left\{
P_2^2(Q)\p_xY
\right\})_j
-2i\sum_{p,q,r=1}^{n}
S_{p,q,r}^{2,j}
\p_x(Q_p\overline{Q_q})\p_xY_r. 
\nonumber
\end{align}
Second, replacing $p$ and $r$ in \eqref{eq:nom2} and 
\eqref{eq:nom3}, 
we deduce 
\begin{align}
\eqref{eq:nom2}+\eqref{eq:nom3}
&=
\sum_{p,q,r=1}^{n}
\dfrac{1}{2}
\left(
2\omega_{r,q,p}^{1,j} 
+\omega_{r,q,p}^{3,j}
+
\omega_{r,p,q}^{4,j}
+
 \omega_{p,r,q}^{4,j}
\right)
\p_x(Q_p\overline{Q_q})\p_xY_r
\nonumber
\\
&\quad 
+
\sum_{p,q,r=1}^{n}
\dfrac{1}{2}
\left(
-\omega_{r,q,p}^{3,j} 
+
\omega_{r,p,q}^{4,j}
+
 \omega_{p,r,q}^{4,j}
\right)
\left(
\p_xQ_p\overline{Q_q}-Q_p\overline{\p_xQ_q}
\right)
\p_xY_r.
\nonumber
\end{align}
Combining them and using \eqref{eq:S3} and \eqref{eq:S4}, 
we have 
\begin{align}
&\eqref{eq:nom1}+\eqref{eq:nom2}+\eqref{eq:nom3}
\nonumber
\\
&=
(P_1^1(Q)\p_x^2Y)_j
+(\p_x\left\{
P_2^2(Q)\p_xY
\right\})_j
+\left(\sum_{k=3}^4 P_k^k(Q)\p_xY\right)_j.
\label{eq:9182}
\end{align}
Third, replacing $p$ and $q$, and using \eqref{eq:S5}, we deduce 
\begin{align}
\eqref{eq:nom4}
&=
\sum_{p,q,r=1}^{n}
\left(
\omega_{q,p,r}^{2,j}+\omega_{q,r,p}^{2,j}
+\omega_{p,q,r}^{3,j} 
\right)
\overline{\p_xY_q}\p_xQ_pQ_r
=
(P_5^5(Q)\p_xY)_j.
\label{eq:9183}
\end{align}
Combing \eqref{eq:9182} and \eqref{eq:9183} shows 
\eqref{eq:b51new} holds for each $j$-th component. This completes the proof.
\end{proof}
We note that none of the conditions (B1)-(B6) have been used to show 
Proposition~\ref{proposition:pF}.
In fact,  (B1)-(B6) under (A1)-(A2) 
ensure the following symmetric or 
skew-symmetric properties of $P_k^{\ell}(Q)$ for $k=2,4,5$ and 
$\ell\in \left\{1,\ldots,5\right\}$. 
\begin{proposition}
\label{proposition:symskewsym}
Suppose that the conditions \textup{(A1)-(A2)} and \textup{(B1)-(B6)} are satisfied. 
Then for any $\ell\in \left\{1,\ldots,5\right\}$,  
\begin{align}
\operatorname{Re}[P_2^{\ell}(Q)v\cdot w]
&=
-\operatorname{Re}[P_2^{\ell}(Q)w\cdot v], 
\label{eq:491}
\\
\operatorname{Re}
[P_4^{\ell}(Q)v\cdot w]
&=
\operatorname{Re}[P_4^{\ell}(Q)w\cdot v],
\label{eq:481}
\\
P_5^{\ell}(Q)v\cdot w
&=
P_5^{\ell}(Q)w\cdot v
\label{eq:391}
\end{align}
for any $v={}^t(v_1,\ldots,v_n)(t,x), w={}^t(w_1,\ldots,w_n)(t,x):[t_1,t_2]\times \TT \to \mathbb{C}^n$.
\end{proposition}
\begin{proof}[Proof of Proposition~\ref{proposition:symskewsym}] 
By Propositions~\ref{proposition:BC} and \ref{proposition:GC}, 
(B1)-(B6) and (G1)-(G5) are equivalent under 
(A1)-(A2).  
Hence $g_{\ell}(p,q,r,j)\in \Gamma_1\cap \Gamma_2$ for 
$\ell \in \left\{1,\ldots,5\right\}$, that is, 
\begin{align}
S_{p,q,r}^{\ell,j}
&=
S_{r,q,p}^{\ell,j}
\quad 
\text{for any $p,q,r,j\in\left\{1,\ldots,n\right\}$ and $\ell\in \left\{1,\ldots,5\right\}$,}
\label{eq:tsu3s}
\\
S_{p,q,r}^{\ell,j}
&=
\overline{S_{j,r,q}^{\ell,p}}
\quad
\text{for any $p,q,r,j\in\left\{1,\ldots,n\right\}$ and $\ell\in \left\{1,\ldots,5\right\}$.}
\label{eq:R34as}
\end{align}
By \eqref{eq:tsu3s} and \eqref{eq:R34as}, we have  
$
\overline{S_{q,p,j}^{\ell, r}}=\overline{S_{j,p,q}^{\ell, r}}=S_{r,q,p}^{\ell,j}=S_{p,q,r}^{\ell,j} 
$ and 
$S_{p,j,r}^{\ell,q}=\overline{S_{q,r,j}^{\ell,p}}
=\overline{S_{j,r,q}^{\ell,p}}=S_{p,q,r}^{\ell,j}$, 
and thus 
\begin{align}
S_{p,q,r}^{\ell ,j}
&=\overline{S_{q,p,j}^{\ell, r}}
\quad
\text{for any $p,q,r,j\in\left\{1,\ldots,n\right\}$ and $\ell\in \left\{1,\ldots,5\right\}$,}
\label{eq:05190} 
\\
S_{p,q,r}^{\ell ,j}
&=S_{p,j,r}^{\ell, q}
\quad
\text{for any $p,q,r,j\in\left\{1,\ldots,n\right\}$ and $\ell\in \left\{1,\ldots,5\right\}$.}
\label{eq:a05190} 
\end{align}
The desired \eqref{eq:491}-\eqref{eq:391} can be obtained 
by using \eqref{eq:05190}-\eqref{eq:a05190}. 
To see this, 
we rewrite as follows: 
$$
P_2^{\ell}(Q)v=\left(\alpha_{jk}\right)v, \quad
P_4^{\ell}(Q)v=\left(\beta_{jk}\right)v, \quad
P_5^{\ell}(Q)v=\left(\gamma_{jk}\right)\overline{v},  
$$
where 
$\left(\alpha_{jk}\right)$, $\left(\beta_{jk}\right)$ and 
$\left(\gamma_{jk}\right)$ 
denote $n\times n$ complex-matrix-valued functions 
whose $(j,k)$-components are respectively 
$\alpha_{jk}$, $\beta_{jk}$ and $\gamma_{jk}$. 
By definition, it is easy to see 
$$
\alpha_{jk}=2i\sum_{p,q=1}^nS_{p,q,k}^{\ell,j}Q_p\overline{Q_q}, 
\quad
\beta_{jk}=i\sum_{p,q=1}^nS_{p,q,k}^{\ell,j}
(\p_xQ_p\overline{Q_q}-Q_p\overline{\p_xQ_q}), 
$$ 
$$
\gamma_{jk}
=-2i\sum_{p,r=1}^n
S_{p,k,r}^{\ell,j}\p_xQ_pQ_r.
$$
Here, by replacing $p$ and $q$,  and by \eqref{eq:05190}, we see
\begin{align}
\overline{\alpha_{kj}}
&=
-2i\sum_{p,q=1}^n\overline{S_{q,p,j}^{\ell,k}}\overline{Q_q} Q_p
=
-2i\sum_{p,q=1}^nS_{p,q,k}^{\ell,j}Q_p\overline{Q_q}
=
-\alpha_{jk}, 
\nonumber
\end{align}
and thus $(\alpha_{jk})^{*}=-\left(\alpha_{jk}\right)$. 
This shows \eqref{eq:491}. 
In the same way as above, by
replacing  $p$ and $q$, and by \eqref{eq:05190},  
we see $(\beta_{jk})^{*}=\left(\beta_{jk}\right)$, 
which shows \eqref{eq:481}. 
On the other hand, by \eqref{eq:a05190}, we have 
$^{t}\left(\gamma_{jk}\right)=\left(\gamma_{jk}\right)$, 
which shows \eqref{eq:391} 
\end{proof}
\begin{remark}
\label{remark:cond929} 
The above proof shows  
each of the properties \eqref{eq:491}-\eqref{eq:391} 
is derived both from $g_{\ell}(p,q,r,j)\in \Gamma_1$ 
and  $g_{\ell}(p,q,r,j)\in\Gamma_2$. 
In other words, applying them for $\ell\in \left\{1,\ldots,5\right\}$ 
is ensured by the assumption 
\textup{(G\textup{$\ell$})} for the corresponding $\ell$. 
\end{remark}
 \section{Uniform estimate for Bona-Smith regularized solutions 
 in $L^{\infty}H^m$.}
  \label{section:local} 
This section aims to obtain uniform estimates for solutions to 
an initial value problem regularized by the 
Bona-Smith approximation.   
Throughout this section, $m$ is supposed be an integer satisfying $m\geqslant 4$.
\par 
To begin with, 
following mainly \cite{BS,ET,II,Mietka,segata},  
we recall the setting of the Bona-smith approximation:
Let $\phi:\RR\to\RR$  be a Schwartz function 
satisfying $0\leqslant \phi(x)\leqslant 1$ on $\RR$ and  
$\phi(x)=1$ on a neighborhood of the origin $x=0$ so that 
$\p_x^k\phi(0)=0$ for all positive integers $k$.
For any $Q_0={}^t(Q_{01},\ldots,Q_{0n})\in H^m(\TT;\mathbb{C}^n)$ and $\ep\in (0,1)$, 
define 
$Q_0^{\ep}: \TT\to \mathbb{C}^n$ by 
$$
Q_0^{\ep}(x)
=
\dfrac{1}{\sqrt{2\pi}}
\sum_{k\in \mathbb{Z}}
\phi(\ep k)
\widehat{Q_0}(k)
e^{ikx} 
\quad
(x\in \TT), 
$$ 
where 
$\widehat{Q_0}={}^t(\widehat{Q_{01}}, \ldots, \widehat{Q_{0n}}):\mathbb{Z}\to \mathbb{C}^n$ 
is defined by 
$$
\widehat{Q_{0j}}(k)
=\dfrac{1}{\sqrt{2\pi}}
\int_{0}^{2\pi}
Q_{0j}(x)e^{-ikx}\,dx
\quad 
(k\in \mathbb{Z}, j\in \left\{1,\ldots,n\right\}).
$$
It follows that $Q_0^{\ep}\in H^{\infty}(\TT;\mathbb{C}^n)$ and 
$Q_0^{\ep}\to Q_0$ in $H^m(\TT;\mathbb{C}^n)$ 
as $\ep \downarrow 0$. 
Moreover, 
\begin{align}
&\|Q_0^{\ep}\|_{H^m}
\leqslant 
\|Q_0\|_{H^m}, 
\label{eq:bs1}
\\
&\|Q_0^{\ep}\|_{H^{m+\ell}}
\leqslant 
C\ep^{-\ell}
\|Q_0\|_{H^m}
\quad 
(\ell=0,1,2,\ldots),
\label{eq:6092}
\\
&\|Q_0^{\ep}-Q_0\|_{H^{m-\ell}}
\leqslant 
C\ep^{\ell}
\|Q_0\|_{H^m}
\quad 
(\ell=0,1,2,\ldots), 
\label{eq:6093}
\end{align} 
where $C$ is a positive 
constant which depends on $m,\ell,\phi$, 
but not on $\ep$. 
The set $\left\{Q_{0}^{\ep}\right\}_{\ep\in (0,1)}$ is called a 
Bona-Smith approximation of $Q_0$.   
\par 
For $\ep\in (0,1)$ and $Q_0\in H^m(\TT;\mathbb{C}^n)$,  
we consider 
the fourth-order parabolic regularized problem:   
\begin{alignat}{2}
 \left(
 \p_t+\ep^5 \p_x^4-iM_a\p_x^4-iM_{\lambda}\p_x^2
 \right)
 Q
  &=F(Q, \p_xQ, \p_x^2Q), 
 \label{eq:bpde}
 \\
  Q(0,x)
   &=
   Q_{0}^{\ep}(x) 
 \label{eq:bdata}
 \end{alignat}
for $Q={}^t(Q_1,\ldots,Q_n):[0,\infty)\times \TT\to \mathbb{C}^n$, 
where $ Q_{0}^{\ep}\in H^{\infty}(\TT;\mathbb{C}^n)$ is given 
by the Bona-Smith approximation of $Q_0$. 
It is not difficult to show there exists  
$T_{\ep}=T(\ep, \|Q_0^{\ep}\|_{H^m})>0$ 
and a unique $Q^{\ep}={}^t(Q^{\ep}_1,\ldots,Q^{\ep}_n)\in C([0,T_{\ep}];H^{\infty}(\TT;\mathbb{C}^n))$ 
solving \eqref{eq:bpde}-\eqref{eq:bdata} 
by the standard contraction mapping argument. 
We omit the detail. 
\par 
The goal of this section is to show the following:
  \begin{proposition}
  \label{proposition:preloc}
  Under the same assumptions as in Theorem~\ref{theorem:lwp}, 
  there exists 
  $T=T(\|Q_0\|_{H^4})>0$ 
  which is independent of $\ep\in (0,1)$ 
  such that $\left\{Q^{\ep}\right\}_{\ep\in (0,1)}$ is bounded in  $L^{\infty}(0,T;H^m(\TT;\mathbb{C}^n))$.
  \end{proposition}
\begin{proof}[Proof of Proposition~\ref{proposition:preloc}]
By Propositions~\ref{proposition:BC} and \ref{proposition:GC}, 
it suffices to show it by assuming (A1)-(A2) and (G1)-(G5).
\par 
For fixed $\ep\in (0,1)$, 
we set $U^{\ep}={}^t(U_1^{\ep},\ldots,U_n^{\ep})=\p_x^mQ^{\ep}$. 
Taking the $(m-1)$-st $x$-derivative 
of 
\eqref{eq:b51new} with $Q$ replaced by $Q^{\ep}$
yields
\begin{align}
&\p_x^m(F(Q^{\ep}, \p_xQ^{\ep}, \p_x^2Q^{\ep}))
=
P_1^1(Q^{\ep})\p_x^2U^{\ep}
+(m-1)\p_x(P_1^1(Q^{\ep}))\p_xU^{\ep}
\nonumber
\\
&\quad\phantom{\p_x^m(F(Q^{\ep}, \p_xQ^{\ep}, \p_x^2Q^{\ep}))}
+\p_x\left\{
P_2^2(Q^{\ep})\p_xU^{\ep}
\right\}
+(m-1)\p_x(P_2^2(Q^{\ep}))\p_xU^{\ep}
\nonumber
\\
&\quad \phantom{\p_x^m(F(Q^{\ep}, \p_xQ^{\ep}, \p_x^2Q^{\ep}))}
+\sum_{k=3}^{5}
P_k^k(Q^{\ep})\p_xU^{\ep}
+\mathcal{R}_{\leqslant m}(Q^{\ep}), 
\nonumber
\\
&\|\mathcal{R}_{\leqslant m}(Q^{\ep})(t)\|_{L^2}
\leqslant 
C(\|Q^{\ep}(t)\|_{H^3})
\|Q^{\ep}(t)\|_{H^m}, 
\label{eq:10161}
\end{align} 
where $\p_x(P_k^{\ell}(Q^{\ep})):=[\p_x, P_k^{\ell}(Q^{\ep})]$ 
and $\left[\cdot,\cdot\right]$ denotes the commutator of two 
operators.  
In particular, we have 
\begin{align}
\p_x(P_1^{\ell}(Q^{\ep}))
&=-P_3^{\ell}(Q^{\ep})+P_5^{\ell}(Q^{\ep}), 
\label{eq:53011}
\\
\p_x(P_2^{\ell}(Q^{\ep}))
&=2P_3^{\ell}(Q^{\ep}) 
\label{eq:53012}  
\end{align} 
for $\ell\in \left\{1,\ldots,5\right\}$.  
Indeed, by the definition of $P_1^{\ell}(Q^{\ep})$, 
\begin{align}
(\p_x(P_1^{\ell}(Q^{\ep}))v)_j
&=
-i\sum_{p,q,r=1}^n
	  	  	  	  S_{p,q,r}^{\ell,j}
\p_x(Q^{\ep}_r\overline{Q^{\ep}_q})v_p
-i\sum_{p,q,r=1}^n
S_{p,q,r}^{\ell,j}
\overline{v_q}(\p_xQ^{\ep}_pQ^{\ep}_r
+Q^{\ep}_p\p_xQ^{\ep}_r).  	 
\nonumber	  	 	  	  
\end{align}
Replacing $p$ and $r$, 
and using \eqref{eq:tsu3s},  
we see 
the first summation term of the right hand side 
is $-(P_3^{\ell}(Q^{\ep})v)_j$ and the second one 
is $(P_5^{\ell}(Q^{\ep})v)_j$, which shows \eqref{eq:53011}.  
On the other hand, \eqref{eq:53012} is immediately follows only 
from the definition of $P_2^{\ell}(Q^{\ep})$ 
and $P_3^{\ell}(Q^{\ep})$. 
Using \eqref{eq:53011} for $\ell=1$ and \eqref{eq:53012} 
for $\ell=2$,   
we obtain 
\begin{align}
\p_x^m(F(Q^{\ep}, \p_xQ^{\ep}, \p_x^2Q^{\ep}))
&=
P_1^1(Q^{\ep})\p_x^2U^{\ep}
+\p_x\left\{
P_2^2(Q^{\ep})\p_xU^{\ep}
\right\}
+
P_{3,m}(Q^{\ep})\p_xU^{\ep}
\nonumber
\\&\quad
+
P_4^4(Q^{\ep})\p_xU^{\ep}
+
P_{5,m}(Q^{\ep})\p_xU^{\ep}
+\mathcal{R}_{\leqslant m}(Q^{\ep}), 
\label{eq:b51}
\end{align}
where 
\begin{align}
P_{3,m}(Q^{\ep})
&=
P_3^3(Q^{\ep})-(m-1)P_3^1(Q^{\ep})+2(m-1)P_3^2(Q^{\ep}), 
\nonumber
\\
P_{5,m}(Q^{\ep})
&=
P_5^5(Q^{\ep})+(m-1)P_5^1(Q^{\ep}).
\nonumber
\end{align}
Therefore, taking the $m$-th $x$-derivative of 
both hand sides of \eqref{eq:bpde} yields 
\begin{align}
\p_tU^{\ep}
&=
P^{\ep}(Q^{\ep})U^{\ep}
+
\mathcal{R}_{\leqslant m}(Q^{\ep}), 
\label{eq:b384}
\end{align}
where 
\begin{align}
P^{\ep}(Q^{\ep})
&=
-\ep^5\,\p_x^4+iM_a\p_x^4+iM_{\lambda}\p_x^2
+P_1^1(Q^{\ep})\p_x^2
+
\p_x\left\{
P_2^2(Q^{\ep})\p_x
\right\}
\nonumber
\\
&\quad
+P_{3,m}(Q^{\ep})\p_x
+P_4^4(Q^{\ep})\p_x
+P_{5,m}(Q^{\ep})\p_x.
\label{eq:loc1}
\end{align}
\par 
We next define  
$\mathcal{E}_m(Q^{\ep})=\mathcal{E}_m(Q^{\ep}(t)):[0,T_{\ep}]\to [0,\infty)$ to satisfy  
\begin{align}
\mathcal{E}_m(Q^{\ep}(t))^2
&=
\|V^{\ep}(t)\|_{L^2}^2+\|Q^{\ep}(t)\|_{H^{m-1}}^2.
\label{eq:Nml}
\end{align}
Here, $V^{\ep}={}^t(V^{\ep}_1,\ldots,V^{\ep}_n): [0,T_{\ep}]\times \TT \to \mathbb{C}^n$ is defined by
\begin{align}
 &V^{\ep}
 :=\p_x^mQ^{\ep}_j
 +M_a^{-1}\Lambda(Q^{\ep})\p_x^{m-2}Q^{\ep}, 
 \label{eq:V_j}
\end{align}
where $M_a^{-1}$ denotes the inverse matrix of $M_a$, 
\begin{align}
\Lambda(Q^{\ep}) 
&=e_1\Lambda_1(Q^{\ep})+e_2\Lambda_2(Q^{\ep})
+e_3\Lambda_3(Q^{\ep}), 
\nonumber
\end{align}
$e_k$ for $k=1,2,3$ are real constants which will be decided later, 
and 
\begin{align}
\Lambda_k(Q^{\ep}) v
={}^{t}((\Lambda_k(Q^{\ep}) v)_1, 
\ldots, (\Lambda_k(Q^{\ep}) v)_n) 
\qquad (k=1,2,3) 
\nonumber
\end{align}
for any  
$v=v(t,x): [0,T_{\ep}]\times \TT\to \mathbb{C}^n$
are defined by 
\begin{align}
&(\Lambda_1(Q^{\ep}) v)_j
=
-\dfrac{1}{2}
\sum_{p,q,r=1}^n
S_{p,q,r}^{1,j}
\left(
v_p\overline{Q^{\ep}_q}
-
Q^{\ep}_p\overline{v_q}
\right)
Q^{\ep}_r, 
\label{eq:b381}
\\
&(\Lambda_2(Q^{\ep}) v)_j
=
-\dfrac{1}{4}
\sum_{p,q,r=1}^n
S_{p,q,r}^{1,j}
Q^{\ep}_p\overline{Q^{\ep}_q}
v_r, 
\label{eq:b382}
\\
&(\Lambda_3(Q^{\ep}) v)_j
=
-\dfrac{1}{4}
\sum_{p,q,r=1}^n
\left\{
S_{p,q,r}^{3,j}-(m-1)(S_{p,q,r}^{1,j}-2S_{p,q,r}^{2,j})
\right\}
Q^{\ep}_p\overline{Q^{\ep}_q}
v_r.
\label{eq:b383}
\end{align}
By the Sobolev embedding, there exist constants $C_1,C_2>0$ which are 
independent of $\ep\in (0,1)$ such that   
\begin{equation}
\frac{\|Q^{\ep}(t)\|_{H^m}^2}{C_1\,(1+\|Q^{\ep}(t)\|_{H^1}^2)}
\leqslant 
\mathcal{E}_m(Q^{\ep}(t))^2
\leqslant 
C_2\,
(1+\|Q^{\ep}(t)\|_{H^1}^2)
\|Q^{\ep}(t)\|_{H^m}^2
\label{eq:eqi9231}
\end{equation}
on $[0,T_{\ep}]$.  
Furthermore, we set 
\begin{equation}
T^{\star}_{\ep}
=
\sup
\left\{
T>0 \ | \ 
\mathcal{E}_4(Q^{\ep}(t))\leqslant 2\mathcal{E}_4(Q^{\ep}(0))
\ \ 
\text{for all}
\ \
t\in[0,T]
\right\}. 
\label{eq:cuttime}
\end{equation}
Combining \eqref{eq:Nml}, 
\eqref{eq:eqi9231}, 
\eqref{eq:cuttime}, 
and $\|Q_0^{\ep}\|_{H^4}\leqslant \|Q_0\|_{H^4}$ 
which follows from \eqref{eq:bs1} for $m=4$, 
we can show that there exist  
$C_k=C_k(\|Q_0\|_{H^4})>0$ for $k=3,4,5$ depending on 
$\|Q_0\|_{H^4}$ but not on $\ep\in (0,1)$ such that  
\begin{align}
\sup_{t\in [0,T_{\ep}^{\star}]}\|Q^{\ep}(t)\|_{H^4}^2
&\leqslant 
C_3(\|Q_0\|_{H^4})    
\label{eq:H3Emn}
\end{align}
and 
\begin{align}
&\frac{\|Q^{\ep}(t)\|_{H^m}^2}{C_4(\|Q_0\|_{H^4})}
\leqslant 
\mathcal{E}_m(Q^{\ep}(t))^2
\leqslant 
C_5(\|Q_0\|_{H^4})\|Q^{\ep}(t)\|_{H^m}^2
\ 
\text{on $[0,T_{\ep}^{\star}]$}.
\label{eq:H3Emm}
\end{align} 
\par 
We next 
investigate the equation satisfied by $V^{\ep}$ on 
$[0,T_{\ep}^{\star}]\times \TT$. 
By \eqref{eq:b384} and \eqref{eq:V_j},   
\begin{align}
\p_tV^{\ep}
&=
P^{\ep}(Q^{\ep})U^{\ep}+\p_t(M_a^{-1}\Lambda(Q^{\ep}) \p_x^{m-2}Q^{\ep})
+
\mathcal{R}_{\leqslant m}(Q^{\ep})
\nonumber
\\
&=
P^{\ep}(Q^{\ep})V^{\ep}
-\left\{P^{\ep}(Q^{\ep})(M_a^{-1}\Lambda(Q^{\ep})\p_x^{m-2}Q^{\ep})
-\p_t(M_a^{-1}\Lambda(Q^{\ep}) \p_x^{m-2}Q^{\ep})
\right\}
\nonumber
\\
&\quad
+
\mathcal{R}_{\leqslant m}(Q^{\ep}).
\label{eq:add9251}
\end{align}
Here, it is easy to see 
$$
\p_t\p_x^{m-2}Q^{\ep}
=
(-\ep^5\p_x^{m+2}+iM_a\p_x^{m+2})Q^{\ep}+r_1, 
$$
where 
$
\|r_1(t)\|_{L^2}\leqslant 
C(\|Q^{\ep}(t)\|_{H^2})
\|Q^{\ep}(t)\|_{H^m}. 
$ 
This shows 
\begin{align}
\p_t(M_a^{-1}\Lambda(Q^{\ep}) \p_x^{m-2}Q^{\ep})
&=
M_a^{-1}\Lambda(Q^{\ep})
(-\ep^5\p_x^{m+2}+iM_a\p_x^{m+2})Q^{\ep}
\nonumber
\\
&\quad 
+O\left(
|\p_tQ^{\ep}|
|Q^{\ep}|
|\p_x^{m-2}Q^{\ep}|
\right)
+r_2, 
\nonumber
\end{align}
where $\|r_2(t)\|_{L^2}
\leqslant 
C(\|Q^{\ep}(t)\|_{H^2})
\|Q^{\ep}(t)\|_{H^m}$.
Since $[M_a,\Lambda(Q^{\ep})]=[aI_n,\Lambda(Q^{\ep})]=0$, 
we see $\Lambda(Q^{\ep})iM_a=\Lambda(Q^{\ep})M_ai=M_a\Lambda(Q^{\ep})i$, 
and hence 
\begin{align}
\p_t(M_a^{-1}\Lambda(Q^{\ep}) \p_x^{m-2}Q^{\ep})
&=
-\ep^5M_a^{-1}\Lambda(Q^{\ep})\p_x^{m+2}Q^{\ep}
+\Lambda(Q^{\ep}) i\p_x^{m+2}Q^{\ep}
\nonumber
\\
&\quad 
+O\left(
|\p_tQ^{\ep}|
|Q^{\ep}|
|\p_x^{m-2}Q^{\ep}|
\right)
+r_2. 
\label{eq:add10281}
\end{align}
On the other hand, by the definition of $P^{\ep}(Q^{\ep})$ given by 
\eqref{eq:loc1}, 
\begin{align}
&P^{\ep}(Q^{\ep})(M_a^{-1}\Lambda(Q^{\ep})\p_x^{m-2}Q^{\ep})
\nonumber
\\
&=
-\ep^5M_a^{-1}\p_x^4\left\{\Lambda(Q^{\ep})\p_x^{m-2}Q^{\ep}\right\}
+i\p_x^4\left\{\Lambda(Q^{\ep})\p_x^{m-2}Q^{\ep}\right\}
+r_3, 
\label{eq:add10282}
\end{align}
where 
$
\|r_3(t)\|_{L^2}
\leqslant 
C(\|Q^{\ep}(t)\|_{H^2})
\|Q^{\ep}(t)\|_{H^m}$.  
Combining \eqref{eq:add9251}-\eqref{eq:add10282}, we obtain
\begin{align}
\p_tV^{\ep}
&=
P^{\ep}(Q^{\ep})V^{\ep}
-
\left[
i\p_x^4,\Lambda(Q^{\ep})\p_x^{-2}
\right]U^{\ep}
+
\ep^5M_a^{-1}
\left[
\p_x^4,\Lambda(Q^{\ep})\p_x^{-2}
\right]U^{\ep}
\nonumber 
\\
&\quad 
+O\left(
|\p_tQ^{\ep}|
|Q^{\ep}|
|\p_x^{m-2}Q^{\ep}|
\right)
+
\mathcal{R}_{\leqslant m}(Q^{\ep})+r_2-r_3, 
\label{eq:c9232}
\end{align}
where we write 
$
\left[
i\p_x^4,\Lambda(Q^{\ep})\p_x^{-2}
\right]U^{\ep}
=
i\p_x^4\left\{\Lambda(Q^{\ep})\p_x^{m-2}Q^{\ep}\right\}
-
\Lambda(Q^{\ep})i\p_x^{m+2}Q^{\ep}
$ and 
$
\left[
\p_x^4,\Lambda(Q^{\ep})\p_x^{-2}
\right]U^{\ep}
=
\p_x^4\left\{\Lambda(Q^{\ep})\p_x^{m-2}Q^{\ep}\right\}
-
\Lambda(Q^{\ep})\p_x^{m+2}Q^{\ep}
$ by noting $U^{\ep}=\p_x^{m}Q^{\ep}$. 
\begin{proposition}
\label{proposition:comm_loc}
\begin{align}
\left[
i\p_x^4, \Lambda_1(Q^{\ep})\p_x^{-2}
\right]U^{\ep}
&=
-
P_1^1(Q^{\ep})\p_x^2U^{\ep}
-\dfrac{1}{2}
\p_x\left\{P_2^1(Q^{\ep})\p_xU^{\ep}\right\}
\nonumber
\\
&\quad 
-P_3^1(Q^{\ep})\p_xU^{\ep}
-2P_5^1(Q^{\ep})\p_xU^{\ep} 
+r_4, 
\label{eq:b386}
\\
\left[
i\p_x^4, \Lambda_2(Q^{\ep})\p_x^{-2}
\right]U^{\ep}
&= 
-P_3^1(Q^{\ep})\p_xU^{\ep}+r_5, 
\label{eq:b387}
\\
\left[
i\p_x^4, \Lambda_3(Q^{\ep})\p_x^{-2}
\right]U^{\ep}
&= 
-P_{3,m}(Q^{\ep})\p_xU^{\ep}+r_6, 
\label{eq:b388}
\end{align}
where $
\|r_k(t)\|_{L^2}
\leqslant 
C(\|Q^{\ep}(t)\|_{H^4})\|Q^{\ep}(t)\|_{H^m}
$ for each $k=4,5,6$. 
\end{proposition}
\begin{proof}[Proof of Proposition~\ref{proposition:comm_loc}]
First,  
\begin{align} 
\left[
i\p_x^4, \Lambda_1(Q^{\ep})\p_x^{-2}
\right]U^{\ep}
&=
K_1+K_2, 
\label{eq:b550}
\end{align}
where 
\begin{align}
K_1&:=
i\p_x^4\left\{
\Lambda_1(Q^{\ep})\p_x^{m-2}Q^{\ep}
\right\}
-i \Lambda_1(Q^{\ep})
\left(
\p_x^{m+2}Q^{\ep}
\right), 
\ 
K_2:=\left[iI_n,\Lambda_1(Q^{\ep})\right]\p_x^{m+2}Q^{\ep}.
\nonumber
\end{align} 
The definition of $\Lambda_1(Q^{\ep})$ and the Leibniz rule yields  
\begin{align}
K_1&=
4i\p_x(\Lambda_1(Q^{\ep}))\p_xU^{\ep}+r_4, 
\nonumber
\end{align}
where 
$
\|r_4(t)\|_{L^2}\leqslant 
C(\|Q^{\ep}(t)\|_{H^4})\|Q^{\ep}(t)\|_{H^m}$.
Here, by a simple computation, 
\begin{align}
(i\p_x(\Lambda_1(Q^{\ep})) v)_j
&=
\dfrac{1}{2}K_{11}+\dfrac{1}{2}K_{12}, 
\nonumber
\end{align}
where 
\begin{align}
K_{11}&:=
i\sum_{p,q,r=1}^n
S_{p,q,r}^{1,j}
\overline{v_q}
\left(
\p_xQ^{\ep}_pQ^{\ep}_r
+
Q^{\ep}_p\p_xQ^{\ep}_r
\right), 
\nonumber
\\
K_{12}&:=
-i
\sum_{p,q,r=1}^n
S_{p,q,r}^{1,j}
v_p\overline{\p_xQ^{\ep}_q}
Q^{\ep}_r
-i
\sum_{p,q,r=1}^n
S_{p,q,r}^{1,j}
v_p\overline{Q^{\ep}_q}
\p_xQ^{\ep}_r.
\nonumber
\end{align}
Both for $K_{11}$ and $K_{12}$, 
by replacing $p$ and $r$ and by using 
\eqref{eq:tsu3s} for $\ell=1$, we see 
$K_{11}
=-(P_5^1(Q^{\ep})v)_j$ and  
$K_{12}
=
-\left(
P_3^1(Q^{\ep})v
\right)_j$. 
Hence we obtain 
\begin{align}
i\p_x(\Lambda_1(Q^{\ep}))
&=
-\dfrac{1}{2}P_3^1(Q^{\ep})
-\dfrac{1}{2}P_5^1(Q^{\ep}), 
\label{eq:I12}
\end{align} 
and thus 
\begin{align}
K_1&=
-2P_3^1(Q^{\ep})\p_xU^{\ep}
-2P_5^1(Q^{\ep})\p_xU^{\ep}
+r_4. 
\label{eq:K1}
\end{align}
On the other hand, by the definition of $\Lambda_1(Q^{\ep})$, 
we deduce 
\begin{align}
\left(
\left[
i\,I_n, \Lambda_1(Q^{\ep})
\right]v
\right)_j
&=
-\dfrac{i}{2}
\sum_{p,q,r=1}^n
S_{p,q,r}^{1,j}
\left(
v_p\overline{Q^{\ep}_q}
-
Q^{\ep}_p\overline{v_q}
\right)
Q^{\ep}_r
\nonumber
\\
&\quad 
+\dfrac{1}{2}
\sum_{p,q,r=1}^n
S_{p,q,r}^{1,j}
\left(
iv_p\overline{Q^{\ep}_q}
-
Q^{\ep}_p\overline{iv_q}
\right)
Q^{\ep}_r
\nonumber
\\
&=
-\left(P_1^1(Q^{\ep})v\right)_j
-i
\sum_{p,q,r=1}^n
S_{p,q,r}^{1,j}
v_p\overline{Q^{\ep}_q}
Q^{\ep}_r.
\nonumber
\end{align}
Applying \eqref{eq:tsu3s} for $\ell=1$ 
and replacing $p$ and $r$ 
in the second term of the right hand side, we obtain  
\begin{align}
\left[
i\,I_n, \Lambda_1(Q^{\ep})
\right]
&=
-P_1^1(Q^{\ep})
-\dfrac{1}{2} P_2^1(Q^{\ep}). 
\label{eq:add10284}
\end{align}
Using \eqref{eq:add10284} and \eqref{eq:53012} for $\ell=1$, 
we have 
\begin{align}
K_2&=
-P_1^1(Q^{\ep})\p_x^2U^{\ep}
-\dfrac{1}{2}\p_x\left\{P_2^1(Q^{\ep}) \p_xU^{\ep}\right\}
+P_3^1(Q^{\ep})\p_xU^{\ep}. 
\label{eq:K2}
\end{align}
Substituting \eqref{eq:K1} and \eqref{eq:K2} 
into \eqref{eq:b550}, 
we derive \eqref{eq:b386}. 
\par 
Second, \eqref{eq:b387} is easy to prove, because $[iI_n,\Lambda_2(Q^{\ep})]=0$. 
Indeed, by a simple computation, 
we see the terms including $\p_x^2U^{\ep}_r$ are canceled out and   
\begin{align}
\left[
i\p_x^4, \Lambda_2(Q^{\ep})\p_x^{-2}
\right]U^{\ep}
&=
4i\p_x(\Lambda_2(Q^{\ep}))\p_xU^{\ep}+r_5, 
\nonumber
\end{align}
where 
\begin{align}
i\p_x(\Lambda_2(Q^{\ep}))
&=
-\dfrac{1}{4}P_3^1(Q^{\ep})
\label{eq:add10285}
\end{align}
follows from the definition of $\Lambda_2(Q^{\ep})$, 
and 
$
\|r_5(t)\|_{L^2}\leqslant 
C(\|Q^{\ep}(t)\|_{H^4})\|Q^{\ep}(t)\|_{H^m}$.
\par 
Third, \eqref{eq:b388} is obtained by the same 
way as we show \eqref{eq:b387}. 
\end{proof}
By Proposition~\ref{proposition:comm_loc}, we have 
\begin{align}
\left[
i\p_x^4,\Lambda(Q^{\ep})\p_x^{-2}
\right]U^{\ep}
&=
-e_1\,P_1^1(Q^{\ep})\p_x^2U^{\ep}
-\dfrac{e_1}{2}
\p_x\left\{P_2^1(Q^{\ep})\p_xU^{\ep}\right\}
\nonumber
\\
&\quad 
-(e_1+e_2)P_3^1(Q^{\ep})\p_xU^{\ep}
-2e_1P_5^1(Q^{\ep})\p_xU^{\ep}
\nonumber
\\
&\quad
-e_3P_{3,m} (Q^{\ep})\p_xU^{\ep}
+e_1r_4+e_2r_5+e_3r_6. 
\label{eq:c9233}           
\end{align}
On the other hand, 
it is not difficult to see 
\begin{align} 
\ep^5M_a^{-1}
\left[
\p_x^4,\Lambda(Q^{\ep})\p_x^{-2}
\right]U^{\ep}
&=
\ep^5
O\left(
|Q^{\ep}||\p_x^2U^{\ep}|
+|\p_xQ^{\ep}||Q^{\ep}||\p_xU^{\ep}|
\right)
+r_7, 
\label{eq:ep9231}
\end{align}
where 
$
\|r_7(t)\|_{L^2}\leqslant 
C(\|Q^{\ep}(t)\|_{H^4})\|Q^{\ep}(t)\|_{H^m}$. 
\par 
Combining \eqref{eq:c9232} and 
\eqref{eq:c9233}-\eqref{eq:ep9231} with the estimate for 
$r_2,\ldots,r_7$ and \eqref{eq:10161}, 
and then substituting 
$U^{\ep}=V^{\ep}-M_a^{-1}\Lambda(Q^{\ep})\p_x^{m-2}Q^{\ep}$, 
we deduce  
\begin{align}
\p_tV^{\ep}
&=
P^{\ep}(Q^{\ep})V^{\ep}
+e_1\,P_1^1(Q^{\ep})\p_x^2V^{\ep}
+\dfrac{e_1}{2}
\p_x\left\{P_2^1(Q^{\ep})\p_xV^{\ep}\right\}
\nonumber
\\
&\quad 
+(e_1+e_2)P_3^1(Q^{\ep})\p_xV^{\ep}
+2e_1P_5^1(Q^{\ep})\p_xV^{\ep}
+e_3P_{3,m} (Q^{\ep})\p_xV^{\ep}
\nonumber
\\
&\quad 
+\ep^5
O\left(
|Q^{\ep}||\p_x^2V^{\ep}|
+|\p_xQ^{\ep}||Q^{\ep}||\p_xV^{\ep}|
\right)
\nonumber
\\
&\quad 
+O\left(
|\p_tQ^{\ep}|
|Q^{\ep}|
|\p_x^{m-2}Q^{\ep}|
\right)
+r_8, 
\nonumber
\end{align}
where   
$
\|r_8(t)\|_{L^2}\leqslant 
C(\|Q^{\ep}(t)\|_{H^4})\|Q^{\ep}(t)\|_{H^m}$.
Substituting \eqref{eq:loc1} into the above, 
and taking $(e_1,e_2,e_3)=(-1,1,-1)$ so that 
all the terms $P_1^1(Q^{\ep})\p_x^2V^{\ep}$, 
$P_3^1(Q^{\ep})\p_xV^{\ep}$ and 
$P_{3,m}(Q^{\ep})\p_xV^{\ep}$ are canceled out, 
we derive   
\begin{align}
\p_tV^{\ep}
&=
\left(-\ep^5\p_x^4+iM_a\p_x^4+iM_{\lambda}\p_x^2\right)V^{\ep}
+
\p_x\left\{
\left(P_2^2(Q^{\ep})-\dfrac{1}{2}P_2^1(Q^{\ep})\right)
\p_xV^{\ep}\right\}
\nonumber
\\
&\quad 
+P_4^4(Q^{\ep})\p_xV^{\ep}
-2P_5^1(Q^{\ep})\p_xV^{\ep}
+P_{5,m}(Q^{\ep})\p_xV^{\ep}
\nonumber
\\
&\quad 
+\ep^5
O\left(
|Q^{\ep}||\p_x^2V^{\ep}|
+|\p_xQ^{\ep}||Q^{\ep}||\p_xV^{\ep}|
\right)
\nonumber
\\
&\quad 
+O\left(
|\p_tQ^{\ep}|
|Q^{\ep}|
|\p_x^{m-2}Q^{\ep}|
\right)
+r_8. 
\label{eq:eqV}
\end{align}
\par
Using \eqref{eq:eqV} and integration by parts, we deduce 
\begin{align}
&\dfrac{1}{2}
\dfrac{d}{dt}
\|V^{\ep}(t)\|_{L^2}^2
=
\operatorname{Re}
\langle
\p_tV^{\ep}(t), V^{\ep}(t)
\rangle
=:-\ep^5E_1+\sum_{k=2}^7E_k, 
\nonumber
\end{align}
where 
\begin{align}
E_1
&=
\|\p_x^2V^{\ep}\|_{L^2}^2
-\Re\langle
O\left(
|Q^{\ep}||\p_x^2V^{\ep}|
+|\p_xQ^{\ep}||Q^{\ep}||\p_xV^{\ep}|
\right), V^{\ep}
\rangle, 
\nonumber
\\
E_2&=
\Re\left\langle
\p_x\left\{
\left(P_2^2(Q^{\ep})-\dfrac{1}{2}P_2^1(Q^{\ep})\right)
\p_xV^{\ep}\right\}, 
V^{\ep}
\right\rangle, 
\nonumber
\\
E_3
&=
\Re\langle
P_4^4(Q^{\ep})\p_xV^{\ep}, 
V^{\ep}
\rangle, 
\quad
E_4=
-2
\Re\langle
P_5^1(Q^{\ep})\p_xV^{\ep}, 
V^{\ep}
\rangle, 
\nonumber
\\
E_5&= 
\Re\langle
P_{5,m}(Q^{\ep})
\p_xV^{\ep}, V^{\ep}
\rangle, 
\quad 
E_6= 
\Re\langle
O\left(
|\p_tQ^{\ep}|
|Q^{\ep}|
|\p_x^{m-2}Q^{\ep}|
\right), V^{\ep}
\rangle, 
\nonumber
\\
E_7&=
\Re\langle
r_8, V^{\ep}
\rangle. 
\nonumber
\end{align}
Here, by using integration by parts and  
\eqref{eq:491} in Proposition~\ref{proposition:symskewsym} for $\ell=1,2$, 
\begin{align}
E_2
&=
-\Re\left\langle
\left(P_2^2(Q^{\ep})-\dfrac{1}{2}P_2^1(Q^{\ep})\right)
\p_xV^{\ep}, 
\p_xV^{\ep}
\right\rangle
=0. 
\label{eq:9241}
\end{align}
By using integration by parts,   
\eqref{eq:481} for $\ell=4$, 
and Sobolev's embedding,  
\begin{align}
E_3
&=
-\dfrac{1}{2}\Re\langle
\p_x(P_4^4(Q^{\ep}))V^{\ep}, 
V^{\ep}
\rangle
\leqslant 
C(\|Q^{\ep}(t)\|_{H^3})\|V^{\ep}(t)\|_{L^2}^2. 
\label{eq:9242}
\end{align} 
In the same way as above, by using    
\eqref{eq:391} for $\ell=1$ 
and that for $\ell=1,5$ respectively, 
\begin{align}
E_4
&=
\Re\langle
\p_x(P_5^1(Q^{\ep}))V^{\ep}, 
V^{\ep}
\rangle
\leqslant 
C(\|Q^{\ep}(t)\|_{H^3})\|V^{\ep}(t)\|_{L^2}^2, 
\label{eq:9243}
\\
E_5
&=
-\dfrac{1}{2}\Re\langle
\p_x(P_{5,m}(Q^{\ep}))V^{\ep}, 
V^{\ep}
\rangle
\leqslant 
C(\|Q^{\ep}(t)\|_{H^3})\|V^{\ep}(t)\|_{L^2}^2. 
\label{eq:9244}
\end{align} 
(We mention that \eqref{eq:391} plays a 
crucial role when we consider the case $n\geqslant 2$, in that 
it is automatically satisfied when $n=1$.)
Noting $\p_tQ^{\ep}=(-\ep^5\p_x^4+iM_a\p_x^4)Q^{\ep}+\cdots\in 
C([0,T_{\ep}^{\star}]; L^2(\TT;\mathbb{C}^n))$
which follows from the assumption $m\geqslant 4$, 
and noting also $\ep\in (0,1)$, 
we use \eqref{eq:eqi9231} to obtain 
\begin{align}
E_6
&\leqslant 
C \|\p_tQ^{\ep}(t)\|_{L^2}
\|Q^{\ep}(t)\|_{L^{\infty}}
\|\p_x^{m-2}Q^{\ep}(t)\|_{L^{\infty}}
\|V^{\ep}(t)\|_{L^2}
\nonumber
\\
&\leqslant 
C(\|Q^{\ep}(t)\|_{H^4})
\mathcal{E}_m(Q^{\ep}(t))^2.  
\label{eq:9245}
\end{align}
By the Young inequality and $\ep\in (0,1)$, 
it is easy to obtain
\begin{align}
-\ep^5E_1
&\leqslant 
-\dfrac{\ep^5}{4}\|\p_x^2V^{\ep}(t)\|_{L^2}^2
+C(\|Q_0\|_{H^4})\|V^{\ep}(t)\|_{L^2}^2. 
\label{eq:9246}
\end{align}
By using \eqref{eq:eqi9231}, it is also easy to see
\begin{align}
E_7
&\leqslant 
\|r_8(t)\|_{L^2}
\|V^{\ep}(t)\|_{L^2}
\leqslant 
C(\|Q^{\ep}(t)\|_{H^4})\mathcal{E}_m(Q^{\ep}(t))^2. 
\label{eq:9247}
\end{align}
Combining \eqref{eq:9241}-\eqref{eq:9247}, we derive 
\begin{align}
&\dfrac{1}{2}
\dfrac{d}{dt}
\|V^{\ep}(t)\|_{L^2}^2
+\dfrac{\ep^5}{4}\|\p_x^2V^{\ep}(t)\|_{L^2}^2
\leqslant 
C(\|Q^{\ep}(t)\|_{H^4})\mathcal{E}_m(Q^{\ep}(t))^2. 
\label{eq:9248}
\end{align}
\par 
On the other hand, by using integration by parts,  
it is not difficult to obtain 
\begin{align}
&\dfrac{1}{2}
\dfrac{d}{dt}
\|Q^{\ep}(t)\|_{H^{m-1}}^2
+\dfrac{\ep^5}{4}\sum_{k=0}^{m-1}
\|\p_x^2Q^{\ep}(t)\|_{H^k}^2
\leqslant 
C(\|Q^{\ep}(t)\|_{H^4})\|Q^{\ep}(t)\|_{H^{m}}^2.
\label{eq:9249}
\end{align}
\par 
Combining \eqref{eq:9248} and \eqref{eq:9249}, 
and using \eqref{eq:H3Emn}-\eqref{eq:H3Emm}, 
we have 
\begin{align}
&
\dfrac{d}{dt}
\mathcal{E}_m(Q^{\ep}(t))^2
\leqslant 
A_m(\|Q_0\|_{H^4})\mathcal{E}_m(Q^{\ep}(t))^2
\quad
\text{on $[0,T_{\ep}^{\star}]$}, 
\nonumber
\end{align}
where $A_m(\|Q_0\|_{H^4})>0$ is a constant depending on $\|Q_0\|_{H^4}$ and $m$ 
but not on $\ep$.  
The Gronwall inequality shows  
\begin{align}
\mathcal{E}_m(Q^{\ep}(t))^2
&\leqslant 
\mathcal{E}_m(Q^{\ep}(0))^2
\exp(A_m(\|Q_0\|_{H^4})t)
\quad
\text{on $[0,T_{\ep}^{\star}]$}.
\nonumber
\end{align}
This inequality for $m=4$ and the definition of $T_{\ep}^{\star}$
implies 
$4\leqslant \exp(A_4(\|Q_0\|_{H^4})T_{\ep}^{\star})$, 
from which we obtain 
\begin{equation}
T_{\ep}^{\star}\geqslant 
\frac{\log 4}{A_4(\|Q_0\|_{H^4})}
=:T>0, 
\label{eq:extime}
\end{equation}
and  
\begin{align}
\sup_{t\in [0,T]}
\mathcal{E}_m(Q^{\ep}(t))^2
&\leqslant 
\mathcal{E}_m(Q^{\ep}(0))^2
\exp(A_m(\|Q_0\|_{H^4})T)
\nonumber. 
\end{align} 
Finally, 
by combining this, \eqref{eq:H3Emm}, \eqref{eq:eqi9231} for $t=0$, 
and  \eqref{eq:bs1}, 
we deduce 
\begin{align}
\sup_{t\in [0,T]}
\|Q^{\ep}(t)\|_{H^m}^2
&\leqslant 
C(T, \|Q_0\|_{H^4}, m)
\|Q_0^{\ep}\|_{H^m}^2
\leqslant 
C(T, \|Q_0\|_{H^4}, m)
\|Q_0\|_{H^m}^2.
\label{eq:H153}
\end{align}
This completes the proof of Proposition~\ref{proposition:preloc}. 
 \end{proof}
 \begin{remark}
 \label{remark:cond9292}
 Our proof of Proposition~\ref{proposition:preloc}  
has applied \eqref{eq:491} for $\ell=1,2$, 
 \eqref{eq:481} for $\ell=4$, 
 and \eqref{eq:391} for $\ell=1,5$. 
 As is noted in Remark~\ref{remark:cond929}, applying them shows 
 our proof actually applies all the conditions \textup{(G1)-(G2)} and 
 \textup{(G4)-(G5)}.  
 Considering Propositions~\ref{proposition:BC} and \ref{proposition:GC}, 
 our proof cannot drop any of the assumptions \textup{(B1)-(B6)} under 
 \textup{(A1)-(A2)}. 
 This is also valid in Sections~\ref{section:BS} and  \ref{section:prooflw} 
 to complete the proof of Theorem~\ref{theorem:lwp}.  
 \end{remark}
 %
\section{Estimate for the difference of approximated solutions} 
\label{section:BS} 
Let 
$\left\{Q_{0}^{\ep}\right\}_{\ep\in (0,1)}$ be the Bona-Smith approximation of 
$Q_0\in H^m(\TT;\mathbb{C}^n)$ with 
an integer $m\geqslant 4$.
Let $Q^{\mu}$ and $Q^{\nu}$ denote the solutions to 
\eqref{eq:bpde}-\eqref{eq:bdata} 
with $(\ep,Q_0^{\ep})$ replaced by $(\mu,Q_0^{\mu})$ and that by $(\nu,Q_0^{\nu})$ 
respectively.  
Proposition~\ref{proposition:preloc} in Section~\ref{section:local} ensures    
both $\left\{Q^{\mu}\right\}_{\mu\in(0,1)}$ and $\left\{Q^{\nu}\right\}_{\nu\in (0,1)}$ 
are uniformly bounded in 
$L^{\infty}(0,T;H^m(\TT;\mathbb{C}^n))$, where 
$T=T(\|Q_0\|_{H^4})>0$ is decided by \eqref{eq:extime} independently of $\mu$ and $\nu$. 
In particular, 
 since \eqref{eq:H153} holds also for $Q^{\mu}$ and $Q_0^{\mu}$, 
\begin{align}
\|Q^{\mu}\|_{C([0,T]:H^m)}
&\leqslant 
C_1(T, \|Q_0\|_{H^4})
\|Q_0^{\mu}\|_{H^m} 
\label{eq:B51}
\\
&\leqslant 
C_2(T,\|Q_0\|_{H^m}), 
\label{eq:B52}
\end{align} 
where $C_1(T, \|Q_0\|_{H^4})$ and $C_2(T,\|Q_0\|_{H^m})$ are positive 
constants which depend also on $m$ but not on $\mu$. 
The same also holds for $Q^{\nu}$. 
\par
The goal of this section is to get the following: 
\begin{proposition}
\label{proposition:cauchy}
There exists a constant $C=C(T,\|Q_0\|_{H^m})>1$ such that 
for all $\mu$ and $\nu$ satisfying 
$0<\mu\leqslant \nu<1$, 
\begin{align}
\|Q^{\mu}-Q^{\nu}\|_{C([0,T];H^1)}
&\leqslant 
C
(\nu^{m-1}+\nu^4),
\label{eq:1cauchy}
\\
\|Q^{\mu}-Q^{\nu}\|_{C([0,T];H^m)}
&\leqslant 
C\left(
\nu^{m-3}+\nu
+\|Q_0^{\mu}-Q_0^{\nu}\|_{H^m}
\right).
\label{eq:mcauchy}
\end{align}  
\end{proposition}
\begin{proof}[Proof of \eqref{eq:1cauchy} in Proposition~\ref{proposition:cauchy}]
For $\mu, \nu$ satisfying 
$0<\mu \leqslant \nu<1$, 
we set $W:=Q^{\mu}-Q^{\nu}$, that is, 
$W={}^t(W_1,\ldots,W_n)$ and
$
W_j
=Q_j^{\mu}-Q_j^{\nu}$ 
for $j\in \left\{1,\ldots,n\right\}$. 
\par 
Since $Q^{\mu}$ and $Q^{\nu}$ satisfy \eqref{eq:bpde} 
with $\ep=\mu$ and that with $\ep=\nu$
respectively, 
\begin{align}
&\left(
 \p_t+\mu^5\p_x^4-iM_a\, \p_x^4-i\, M_{\lambda}\p_x^2
 \right)\p_xW
\nonumber
\\
&=
\p_x\left\{F(Q^{\mu},\p_xQ^{\mu},\p_xQ^{\mu})\right\}
 -\p_x\left\{F(Q^{\nu},\p_xQ^{\nu},\p_xQ^{\nu})\right\}
 +(\nu^5-\mu^5)\p_x^4(\p_xQ^{\nu}). 
 \nonumber
\end{align}
Since Proposition~\ref{proposition:pF} for $Q$ is 
also valid for $Q^{\mu}$ and $Q^{\nu}$, 
we can easily obtain 
\begin{align}
&\left(
 \p_t+\mu^5\p_x^4-iM_a\, \p_x^4-i\, M_{\lambda}\p_x^2
 \right)\p_xW
\nonumber
\\
&=
(\nu^5-\mu^5)\p_x^5Q^{\nu}
+
 P_1^1(Q^{\mu})\p_x^3W
 +
 \p_x\left\{
 P_2^2(Q^{\mu})\p_x^2W
 \right\}
 +\sum_{k=3}^{5}
 P_k^k(Q^{\mu})\p_x^2W
 +r_1, 
 \nonumber
 \end{align}
where $P_k^{\ell}(Q^{\mu})$ is defined by \eqref{eq:b511}-\eqref{eq:b515} 
with $Q$ replaced by $Q^{\mu}$, 
and 
$$
\|r_1(t)\|_{L^2}
\leqslant 
C(\|Q^{\mu}(t)\|_{H^2}, \|Q^{\nu}(t)\|_{H^3}) \|W(t)\|_{H^1}.
$$
It follows that 
\begin{align}
&\dfrac{1}{2}\dfrac{d}{dt}
 \|\p_xW(t)\|_{L^2}^2
 +\mu^5\|\p_x^3W(t)\|_{L^2}^2
 \nonumber
 \\
 &\leqslant
 \Re
 \langle
 P_1^1(Q^{\mu})\p_x^3W,\p_xW
 \rangle
 +
\Re
 \langle
 \p_x\{P_2^2(Q^{\mu})\p_x^2W\},\p_xW
 \rangle
 \nonumber
 \\
 &\quad
 +\sum_{k=3}^5
 \Re\langle
 P_k^k(Q^{\mu})\p_x^2W,\p_xW
 \rangle
 +(\nu^5-\mu^5)\Re
 \langle
 \p_x^5Q^{\nu}, \p_xW
 \rangle
 \nonumber
 \\
 &\quad
 +C(\|Q^{\mu}(t)\|_{H^2}, \|Q^{\nu}(t)\|_{H^3})\|W(t)\|_{H^1}^2.
 \nonumber
 \end{align}
Here,   
by using $\p_x(P_1^1(Q^{\mu}))=-P_3^1(Q^{\mu})+P_5^1(Q^{\mu})$ 
and integration by parts, 
\begin{align}
\Re
  \langle
  P_1^1(Q^{\mu})\p_x^3W,\p_xW
  \rangle
&
=
  -\Re
    \langle
    P_1^1(Q^{\mu})\p_x^2W,\p_x^2W
    \rangle
    +\Re
      \langle
      P_3^1(Q^{\mu})\p_x^2W,\p_xW
      \rangle
 \nonumber
 \\
 &\quad
          -\Re
            \langle
            P_5^1(Q^{\mu})\p_x^2W,\p_xW
            \rangle.
            \nonumber
  \end{align}
  Substitution of  it into the above and the 
  same way as we obtain \eqref{eq:9241}-\eqref{eq:9243} 
  yields 
\begin{align}
 &\dfrac{1}{2}\dfrac{d}{dt}
 \|\p_xW(t)\|_{L^2}^2
 +\mu^5\|\p_x^3W(t)\|_{L^2}^2
 \nonumber
 \\
 &\leqslant 
  -\Re
     \langle
     P_1^1(Q^{\mu})\p_x^2W,\p_x^2W
     \rangle
+\Re
      \langle
      P_3^1(Q^{\mu})\p_x^2W,\p_xW
      \rangle     
\nonumber
\\&\quad 
+\Re
      \langle
      P_3^3(Q^{\mu})\p_x^2W,\p_xW
      \rangle
+(\nu^5-\mu^5)\|\p_x^5Q^{\nu}(t)\|_{L^2}\|W(t)\|_{H^1} 
\nonumber
\\
&\quad 
+C(\|Q^{\mu}(t)\|_{H^3}, \|Q^{\nu}(t)\|_{H^3})\|W(t)\|_{H^1}^2. 
\label{eq:92414}
 \end{align}
\par 
We define  
$\mathcal{E}_1^{\mu, \nu}(W)
=\mathcal{E}_1^{\mu, \nu}(W(t))$ 
by 
\begin{align}
\mathcal{E}_1^{\mu, \nu}(W(t))
&=
\frac{1}{2}\|\p_xW(t)\|_{L^2}^2
+A\, \|W(t)\|_{L^2}^2
-\Re\langle
W(t), M_a^{-1}\Lambda(Q^{\mu})W(t)
\rangle, 
\label{eq:5313}
\end{align}
where 
$\Lambda(Q^{\mu}) 
=-\Lambda_1(Q^{\mu})+\Lambda_2(Q^{\mu})
-\Lambda_3(Q^{\mu})$  
and 
\begin{align}
\Lambda_k(Q^{\mu}) v
={}^{t}((\Lambda_k(Q^{\mu}) v)_1, 
\ldots, (\Lambda_k(Q^{\mu}) v)_n) 
\qquad (k=1,2,3) 
\nonumber
\end{align}
for any  
$v=v(t,x): [0,T]\times \TT\to \mathbb{C}^n$
are defined by 
\begin{align}
&(\Lambda_1(Q^{\mu}) v)_j
=
-\dfrac{1}{2}
\sum_{p,q,r=1}^n
S_{p,q,r}^{1,j}
\left(
v_p\overline{Q^{\mu}_q}
-
Q^{\mu}_p\overline{v_q}
\right)
Q^{\mu}_r, 
\label{eq:u381}
\\
&(\Lambda_2(Q^{\mu}) v)_j
=
-\dfrac{1}{4}
\sum_{p,q,r=1}^n
S_{p,q,r}^{1,j}
Q^{\mu}_p\overline{Q^{\mu}_q}
v_r, 
\label{eq:u382}
\\
&(\Lambda_3(Q^{\mu}) v)_j
=
-\dfrac{1}{4}
\sum_{p,q,r=1}^n
S_{p,q,r}^{3,j}
Q^{\mu}_p\overline{Q^{\mu}_q}
v_r.
\label{eq:u383}
\end{align}
Each $\Lambda_k(Q^{\mu})$ can be regarded as 
$\Lambda_k(Q^{\ep})$ defined by \eqref{eq:b381}-\eqref{eq:b383} 
with $(\ep,m)$ replaced by $(\mu,1)$. 
Moreover, $A=2C_2(T, \|Q_0\|_{H^m})$ and  
$C_2(T, \|Q_0\|_{H^m})>1$ is a constant that is taken to satisfy
\begin{align}
\left|\Re\langle
W(t), M_a^{-1}\Lambda(Q^{\mu})W(t)
\rangle\right|
&\leqslant 
C_1(\|Q^{\mu}(t)\|_{H^1})\|W(t)\|_{L^2}^2 
\nonumber
\\
&\leqslant 
C_2(T, \|Q_0\|_{H^m})\|W(t)\|_{L^2}^2, 
\nonumber
\end{align}
which is possible by \eqref{eq:B52}. 
Taking $A$ in this manner ensures 
\begin{align}
\frac{1}{2}\|W(t)\|_{H^1}^2
&\leqslant 
\mathcal{E}_1^{\mu, \nu}(W(t))
\leqslant 
3C_2(T, \|Q_0\|_{H^m})\|W(t)\|_{H^1}^2
\quad 
\text{on $[0,T]$.}
\label{eq:92415}
\end{align}
\begin{remark} 
\label{remark:ME}
One may expects $\p_xW+M_a^{-1}\Lambda(Q^{\mu})\p_x^{-1}W$ 
as an analogy of  \eqref{eq:V_j} 
plays an essential role in our proof.  
A formal calculation using integration by parts implies  
\begin{align}
&\left\|
\p_xW+M_a^{-1}\Lambda(Q^{\mu})\p_x^{-1}W
\right\|_{L^2}^2
\nonumber
\\
&=
\|\p_xW\|_{L^2}^2
-2\Re \langle
W, M_a^{-1}\Lambda(Q^{\mu})W
\rangle
+\text{(terms including $\p_x^{-1}W$).}
\nonumber 
\end{align}
We define $\mathcal{E}_1^{\mu,\nu}(W)$ 
by taking out only the 
first two terms of the right hand side of it which makes sense. 
The addition of $A\|W\|_{L^2}^2$ ensures the 
positivity of  $\mathcal{E}_1^{\mu,\nu}(W)$ and \eqref{eq:92415}, 
the part of which is motivated by the idea of the modified energy in 
\cite{Kwon,segata}.      
\end{remark}
\par 
We next compute 
\begin{align}
F:&=-\dfrac{d}{dt}
\Re\langle
W(t), M_a^{-1}\Lambda(Q^{\mu})W(t)
\rangle
\nonumber
\\
&=
-\Re\langle
\p_tW(t), M_a^{-1}\Lambda(Q^{\mu})W(t)
\rangle
-\Re\langle
W(t), M_a^{-1}\Lambda(Q^{\mu})\p_tW(t)
\rangle
\nonumber
\\&\quad
+
O\left(
|\p_tQ^{\mu}||Q^{\mu}||W|
\right). 
\nonumber
\end{align}
Noting 
\begin{align}
\p_tW
&=
(-\mu^5\p_x^4+iM_a\p_x^4+iM_{\lambda}\p_x^2)W
\nonumber\\
&\quad 
+(\nu^5-\mu^5)\p_x^4Q^{\nu}
+O(|Q^{\mu}|^2|\p_x^2W|)
+r_2
\label{eq:Wt}
\end{align}
where
$$\|r_2(t)\|_{L^2}
\leqslant 
C(\|Q^{\mu}(t)\|_{H^2}, \|Q^{\nu}(t)\|_{H^3})
\|W(t)\|_{H^1}, 
$$
we  use  integration by parts and the Sobolev embedding, 
which yields
\begin{align}
F&\leqslant 
\mu^5F_1+F_2+F_3+F_4+
C(\|Q^{\mu}(t)\|_{H^2}, \|Q^{\nu}(t)\|_{H^3})
\|W(t)\|_{H^1}^2, 
\label{eq:9250}
\end{align}
where 
\begin{align}
F_1&:=
\Re\langle
\p_x^4W, M_a^{-1}\Lambda(Q^{\mu})W
\rangle
+\Re\langle
W, M_a^{-1}\Lambda(Q^{\mu})\p_x^4W
\rangle, 
\nonumber
\\
F_2&:=
-\Re\langle
iM_a\p_x^4W, M_a^{-1}\Lambda(Q^{\mu})W
\rangle
-\Re\langle
W, M_a^{-1}\Lambda(Q^{\mu})iM_a\p_x^4W
\rangle
\nonumber, 
\\
F_3&:=
-(\nu^5-\mu^5)
\left\{\Re\langle
\p_x^4Q^{\nu}, M_a^{-1}\Lambda(Q^{\mu})W
\rangle
+\Re\langle
W, M_a^{-1}\Lambda(Q^{\mu})\p_x^4Q^{\nu}
\rangle
\right\}, 
\nonumber
\\
F_4&:=
C\|\p_tQ^{\mu}(t)\|_{L^2}
\|Q^{\mu}(t)\|_{L^{\infty}}
\|W(t)\|_{L^{\infty}}
\|W(t)\|_{L^2}. 
\nonumber
\end{align}
Here, by \eqref{eq:B52},   
\begin{align}
F_3&
\leqslant 
C(T,\|Q_0\|_{H^m})(\nu^5-\mu^5)
\|\p_x^4Q^{\nu}(t)\|_{L^2}
\|W(t)\|_{L^2}. 
\label{eq:9251}
\end{align}
Since $Q^{\mu}$ satisfies \eqref{eq:bpde} with $\ep=\mu\in (0,1)$ and $m\geqslant 4$,    
\begin{align}
F_4&
\leqslant 
C(\|Q^{\mu}(t)\|_{H^4})
\|W(t)\|_{H^1}^2.
\label{eq:9252}
\end{align}
By integration by parts and the Young inequality, 
\begin{align}
\mu^5F_1
&\leqslant 
\dfrac{\mu^5}{4}\|\p_x^3W(t)\|_{L^2}^2
+C(\|Q^{\mu}(t)\|_{H^2})\|W(t)\|_{H^1}^2. 
\label{eq:9253}
\end{align}
Recalling $[M_a,\Lambda(Q^{\mu})]=0$ follows 
from the assumption on $M_a$, we see 
\begin{align}
M_a^{-1}\Lambda(Q^{\mu})iM_a\p_x^4
&=
M_a^{-1}M_a\Lambda(Q^{\mu})i\p_x^4W
=\Lambda(Q^{\mu})i\p_x^4W, 
\nonumber
\end{align}
and hence 
\begin{align}
F_2
&=
-\Re\langle
i\p_x^4W, \Lambda(Q^{\mu})W
\rangle
-\Re\langle
W, \Lambda(Q^{\mu})i\p_x^4W
\rangle
\nonumber
\end{align}
Further, 
repeated use of integration by parts and the Sobolev embedding 
leads to
\begin{align}
F_2
&\leqslant 
F_{21}+2F_{22}+C(\|Q^{\mu}(t)\|_{H^4})\|W(t)\|_{H^1}^2, 
\label{eq:9254}
\end{align} 
where 
\begin{align}
F_{21}&:=-\Re\langle
i\p_x^2W, \Lambda(Q^{\mu})\p_x^2W
\rangle
-\Re\langle
\p_x^2W, \Lambda(Q^{\mu})i\p_x^2W
\rangle, 
\nonumber
\\
F_{22}&:=\Re\langle
i\p_xW, \p_x(\Lambda(Q^{\mu}))\p_x^2W
\rangle
-\Re\langle
\p_xW, \p_x(\Lambda(Q^{\mu}))i\p_x^2W
\rangle. 
\nonumber
\end{align}
By a simple computation of complex numbers,    
\begin{align}
F_{21}
&=
\Re\langle 
i\Lambda(Q^{\mu})\p_x^2W, 
\p_x^2W
\rangle
-\Re\langle 
\Lambda(Q^{\mu})i\p_x^2W,
\p_x^2W
\rangle
\nonumber
\\
&=
\Re\langle
\left[
i\,I_n, \Lambda(Q^{\mu})
\right]\p_x^2W,\p_x^2W
\rangle. 
\nonumber
\end{align}
By \eqref{eq:add10284} with $Q^{\ep}$ replaced by $Q^{\mu}$ 
and $[iI_n, \Lambda_2(Q^{\mu})]=[iI_n, \Lambda_3(Q^{\mu})]=0$, 
and by 
$\Lambda(Q^{\mu})
=-\Lambda_1(Q^{\mu})+\Lambda_2(Q^{\mu})-\Lambda_3(Q^{\mu})$, 
we obtain 
\begin{align}
F_{21}
&=
\Re\langle
P_1^1(Q^{\mu})\p_x^2W, \p_x^2W
\rangle,
\label{eq:II1}
\end{align} 
where we also used $\Re\langle
P_2^1(Q^{\mu})\p_x^2W, \p_x^2W
\rangle=0$ which follows from \eqref{eq:491} for $\ell=1$ with $Q$ 
replaced by $Q^{\mu}$. 
In the same way as above, a simple computation of complex numbers shows 
\begin{align}
F_{22}
&=
-\Re\langle 
i \p_x(\Lambda(Q^{\mu}))\p_x^2W, \p_xW
\rangle
-\Re\langle
\p_x(\Lambda(Q^{\mu}))i\p_x^2W, \p_xW
\rangle
\nonumber
\\
&=
-\Re
\langle
\left\{ 
i \p_x(\Lambda(Q^{\mu}))
+
\p_x(\Lambda(Q^{\mu}))i 
\right\}
\p_x^2W, \p_xW
\rangle. 
\nonumber
\end{align}
Here, since 
\begin{align}
\p_x(\Lambda(Q^{\mu}))i
&=
i\p_x(\Lambda_1(Q^{\mu}))-\left[iI_n,\p_x(\Lambda(Q^{\mu}))\right]
=
i\p_x(\Lambda_1(Q^{\mu}))-\p_x\left(
\left[iI_n,\Lambda(Q^{\mu})\right]\right), 
\nonumber
\end{align}
we have 
$$
i \p_x(\Lambda(Q^{\mu}))
+
\p_x(\Lambda(Q^{\mu}))i
=
2i\p_x(\Lambda_1(Q^{\mu}))-\p_x\left(
\left[iI_n,\Lambda(Q^{\mu})\right]\right). 
\nonumber
$$
Further, by \eqref{eq:I12}, \eqref{eq:add10284} and 
\eqref{eq:53011}-\eqref{eq:53012} for $\ell=1$ with 
$Q^{\ep}$ replaced by $Q^{\mu}$,  
\begin{align}
i \p_x(\Lambda_1(Q^{\mu}))
+
\p_x(\Lambda_1(Q^{\mu}))i
&=
-P_3^1(Q^{\mu})-P_5^1(Q^{\mu})
+\p_x(P_1^1(Q^{\mu}))+\dfrac{1}{2}\p_x(P_2^1(Q^{\mu}))
\nonumber
\\
&=
-P_3^1(Q^{\mu}). 
\nonumber
\end{align}
On the other hand, by \eqref{eq:add10285} and 
$[iI_n,\Lambda_2(Q^{\mu})]=0$,   
\begin{align}
&
i \p_x(\Lambda_2(Q^{\mu}))
+
\p_x(\Lambda_2(Q^{\mu}))i
=
-\dfrac{1}{2}
P_3^1(Q^{\mu}). 
\nonumber
\end{align}
By the same reason as above for $\Lambda_2(Q^{\ep})$,  
\begin{align}
&
i \p_x(\Lambda_3(Q^{\mu}))
+
\p_x(\Lambda_3(Q^{\mu}))i
=
-\dfrac{1}{2}
P_3^3(Q^{\mu}). 
\nonumber
\end{align} 
Combining them, 
we obtain  
\begin{align}
i \p_x(\Lambda(Q^{\mu}))
+
\p_x(\Lambda(Q^{\mu}))i
&=
\dfrac{1}{2}P_3^1(Q^{\mu})
+
\dfrac{1}{2}P_3^3(Q^{\mu}),  
\nonumber
\end{align}
which yields 
\begin{equation}
2F_{22}
=-\Re\langle
P_3^1(Q^{\mu})\p_x^2W, 
\p_xW
\rangle
-\Re\langle
P_3^3(Q^{\mu})\p_x^2W, 
\p_xW
\rangle. 
\label{eq:II2} 
\end{equation}
Substituting \eqref{eq:II1} and \eqref{eq:II2} into 
\eqref{eq:9254}, 
we get 
\begin{align}
F_2&\leqslant 
\Re\langle
P_1^1(Q^{\mu})\p_x^2W, \p_x^2W
\rangle
-\Re\langle
P_3^1(Q^{\mu})\p_x^2W, 
\p_xW
\rangle
\nonumber
\\
&\quad 
-\Re\langle
P_3^3(Q^{\mu})\p_x^2W, 
\p_xW
\rangle
+C(\|Q^{\mu}(t)\|_{H^4})\|W(t)\|_{H^1}^2. 
\label{eq:9255}
\end{align} 
Combining 
\eqref{eq:9250}-\eqref{eq:9253}, and \eqref{eq:9255}, 
we obtain 
\begin{align}
F
&\leqslant 
\dfrac{\mu^5}{4}\|\p_x^3W(t)\|_{L^2}^2
+
\Re\langle
P_1^1(Q^{\mu})\p_x^2W, \p_x^2W
\rangle
\nonumber
\\
&\quad 
-\Re\langle
P_3^1(Q^{\mu})\p_x^2W, 
\p_xW
\rangle
-\Re\langle
P_3^3(Q^{\mu})\p_x^2W, 
\p_xW
\rangle
\nonumber
\\
&\quad 
+C(T,\|Q_0\|_{H^m})(\nu^5-\mu^5)
\|\p_x^4Q^{\nu}(t)\|_{L^2}
\|W(t)\|_{L^2}
\nonumber
\\
&\quad 
+C(\|Q^{\mu}(t)\|_{H^4}, \|Q^{\nu}(t)\|_{H^3})\|W(t)\|_{H^1}^2.
\label{eq:9256}
\end{align}
\par 
In addition, 
by \eqref{eq:Wt}, 
it is easy to obtain 
\begin{align}
\dfrac{d}{dt}
\left[\|W(t)\|_{L^2}^2\right]
&\leqslant  
C(\|Q^{\mu}(t)\|_{H^2}, \|Q^{\nu}(t)\|_{H^3})
\|W(t)\|_{H^1}^2
\nonumber
\\
&\quad 
+ 
(\nu^5-\mu^5)\|\p_x^4Q^{\nu}(t)\|_{L^2}
\|W(t)\|_{L^2}.
\label{eq:9257}
\end{align}
\par 
Adding \eqref{eq:92414},   
\eqref{eq:9256}-\eqref{eq:9257}, 
and using \eqref{eq:92415},   
we derive  
\begin{align}
\dfrac{d}{dt}
\mathcal{E}_1^{\mu,\nu}(W(t))
&\leqslant 
C(T,\|Q_0\|_{H^m})(\nu^5-\mu^5)
\|Q^{\nu}(t)\|_{H^5}\|W(t)\|_{H^1}
\nonumber
\\
&\quad 
+C(T, \|Q_0\|_{H^m}, \|Q^{\mu}(t)\|_{H^4}, \|Q^{\nu}(t)\|_{H^3})
\mathcal{E}_1^{\mu,\nu}(W(t)). 
\nonumber
\end{align}
Further, 
by applying \eqref{eq:B51} for $m=5$ 
and  \eqref{eq:6092} for $(\ep,m,\ell)=(\nu,4,1)$,   
$$
\|Q^{\nu}\|_{C([0,T];H^5)}
\leqslant 
C(T,\|Q_0\|_{H^4})\nu^{-1}\|Q_0\|_{H^4}, 
$$ 
from which and \eqref{eq:92415}, it follows that 
\begin{align}
&(\nu^5-\mu^5)
\|Q^{\nu}(t)\|_{H^5}\|W(t)\|_{H^1}
\leqslant 
C(T,\|Q_0\|_{H^m})
\nu^4
\left(\mathcal{E}^{\mu,\nu}_1(W(t))\right)^{1/2}.
\nonumber
\end{align}
From this and \eqref{eq:B52}, we obtain    
\begin{align}
\dfrac{d}{dt}
\mathcal{E}_1^{\mu,\nu}(W(t))
&\leqslant 
C(T, \|Q_0\|_{H^m})
\left\{
\mathcal{E}_1^{\mu,\nu}(W(t))
+
\nu^4
\left(\mathcal{E}^{\mu,\nu}_1(W(t))\right)^{1/2}
\right\}. 
\label{eq:u385}
\end{align} 
The Gronwall inequality and \eqref{eq:92415} implies   
\begin{align}
\|W(t)\|_{H^1}
&\leqslant
C(T,\|Q_0\|_{H^m})
(\|W(0)\|_{H^1}+\nu^4). 
\nonumber
\end{align}
By the triangle inequality 
$\|W(0)\|_{H^1}=
\|Q^{\mu}_0-Q^{\nu}_0\|_{H^1}
\leqslant 
\|Q^{\mu}_0-Q_0\|_{H^1}
+\|Q_0-Q^{\nu}_0\|_{H^1}$, 
\eqref{eq:6093} with $\ell=m-1$ and $\ep=\mu,\nu$, 
and by $0<\mu\leqslant \nu<1$, 
this shows 
\begin{align}
\|W\|_{C([0,T];H^1)}
&
\leqslant 
2C(T,\|Q_0\|_{H^m})(\nu^{m-1}+\nu^4), 
\label{eq:90312}
\end{align}
which is the desired \eqref{eq:1cauchy}.
\end{proof}
\begin{proof}[Proof of \eqref{eq:mcauchy} in Proposition~\ref{proposition:cauchy}]
For $\mu, \nu$ with 
$0<\mu \leqslant \nu<1$, 
we set $W:=Q^{\mu}-Q^{\nu}: [0,T]\times \TT\to \mathbb{C}^n$ again and 
define 
\begin{align}
 &Z^{m}
 :=\p_x^mW+M_a^{-1}\Lambda(Q^{\mu})\p_x^{m-2}W 
 \label{eq:Z_j}
 \end{align} 
where 
 $
 \Lambda(Q^{\mu})
 =-\Lambda_1(Q^{\mu})
 +\Lambda_2(Q^{\mu})
 -\Lambda_3(Q^{\mu})
 $
 and $\Lambda_k(Q^{\mu})$  for $k=1,2,3$ are respectively defined 
 by \eqref{eq:b381}-\eqref{eq:b383} with $\ep$ replaced by $\mu$. 
Further, we define 
$\mathcal{E}^{\mu,\nu}_m(W)
=\mathcal{E}^{\mu,\nu}_m(W(t)):[0,T]\to [0,\infty)$ 
to satisfy 
\begin{align}
\mathcal{E}^{\mu,\nu}_m(W(t))^2
&=
\|Z^m(t)\|_{L^2}^2+\|W(t)\|_{H^{m-1}}^2. 
\label{eq:Ekl}
\end{align}
By a similar argument to obtain \eqref{eq:eqi9231} 
and by \eqref{eq:B52}, there exists  
$C_3(T,\|Q_0\|_{H^m})>0$ 
which is independent of $\mu$ and $\nu$ such that 
 \begin{align}
 &\frac{\|W(t)\|_{H^m}^2}{C_3(T,\|Q_0\|_{H^m})}
 \leqslant 
 \mathcal{E}^{\mu,\nu}_m(W(t))^2
 \leqslant 
 C_3(T,\|Q_0\|_{H^m})\|W(t)\|_{H^m}^2
 \quad 
 \text{on $[0,T]$.}
 \label{eq:H4Em}
 \end{align}
\par 
Subtracting \eqref{eq:b384} with $\ep=\nu$ from 
that with $\ep=\mu$, 
and using \eqref{eq:B52}, 
we deduce 
\begin{align}
\p_t\p_x^mW
&=
P(Q^{\mu})\p_x^mW
+(\nu^5-\mu^5) \p_x^{m+4}Q^{\nu}
\nonumber
\\
&\quad 
+O\left(
|W||\p_x^{m+2}Q^{\nu}|
\right)
+
O\left(
(|\p_xW|+|W|)|\p_x^{m+1}Q^{\nu}|
\right)
+\mathcal{R}_{\leqslant m}^{1}, 
\label{eq:9258}
\end{align}
where $P(Q^{\mu})$ is given by \eqref{eq:loc1} with 
$\ep=\mu$ and 
$$
\|\mathcal{R}_{\leqslant m}^{1}(t)\|_{L^2}
\leqslant 
C(T,\|Q_0\|_{H^m})\|W(t)\|_{H^m}.
$$
In the same way as we obtain \eqref{eq:c9232}, 
we use
\eqref{eq:9258} 
and $[M_a, \Lambda(Q^{\mu})]=0$, which shows 
\begin{align}
&\p_tZ^m
=
P(Q^{\mu})Z^m
-
\left[
i\p_x^4,\Lambda(Q^{\mu})\p_x^{-2}
\right]\p_x^mW
+
\mu^5M_a^{-1}
\left[
\p_x^4,\Lambda(Q^{\mu})\p_x^{-2}
\right]\p_x^mW
\nonumber 
\\
&\quad \phantom{\p_tZ^m}
+(\nu^5-\mu^5)
\left\{
\p_x^{m+4}Q^{\nu}+
M_a^{-1}\Lambda(Q^{\mu})\p_x^{m+2}Q^{\nu}
\right\}
\nonumber
\\
&\quad \phantom{\p_tZ^m}
+O\left(
|\p_tQ^{\mu}|
|Q^{\mu}|
|\p_x^{m-2}W|
\right)
+
O\left(
|W||\p_x^{m+2}Q^{\nu}|
\right)
\nonumber
\\
&\quad \phantom{\p_tZ^m}
+
O\left(
(|\p_xW|+|W|)|\p_x^{m+1}Q^{\nu}|
\right)
+\mathcal{R}_{\leqslant m}^{2},  
\nonumber 
\\
&\|\mathcal{R}_{\leqslant m}^{2}(t)\|_{L^2}
\leqslant 
C(T,\|Q_0\|_{H^m})\|W(t)\|_{H^m}.
\nonumber
\end{align}
Here, $\left[
i\p_x^4,\Lambda(Q^{\mu})\p_x^{-2}
\right]\p_x^mW$ is given by \eqref{eq:c9233} 
with $(\ep,U^{\ep})$ replaced by $(\mu,\p_x^mW)$
under $(e_1,e_2,e_3)=(-1,1,-1)$, 
which cancels out $P_1^1(Q^\mu)\p_x^2Z^m$ and 
$P_{3,m}(Q^{\mu})\p_xZ^m$ included in $P(Q^{\mu})Z^m$. 
By noting this and by applying \eqref{eq:491}-\eqref{eq:391} in 
Proposition~\ref{proposition:symskewsym} with $Q$ replaced by 
$Q^{\mu}$, it is obvious to obtain
\begin{align}
\frac{1}{2}\dfrac{d}{dt}
\|Z^m(t)\|_{L^2}^2
&\leqslant 
C(T,\|Q_0\|_{H^m})\mathcal{E}^{\mu,\nu}_m(W(t))^2
\nonumber
\\
&\quad 
+C(T,\|Q_0\|_{H^m})
\|W(t)\|_{H^1}
\|Q^{\nu}(t)\|_{H^{m+2}}
\|Z^m(t)\|_{L^2}
\nonumber
\\
&\quad
+C(T,\|Q_0\|_{H^m})
(\nu^5-\mu^5)
\|Q^{\nu}(t)\|_{H^{m+4}}
\|Z^m(t)\|_{L^2}.
\label{eq:9016}
\end{align} 
Further, by combining \eqref{eq:B51} with $m$ replaced by $m+j$ for 
$j=1,2,\ldots$ and \eqref{eq:6092},  
\begin{align}
\|Q^{\nu}\|_{C([0,T];H^{m+j})}
&\leqslant 
C(T,\|Q_0\|_{H^m})
\nu^{-j}
\quad 
(j=1,2,\ldots).
\label{eq:90162}
\end{align}
Combining \eqref{eq:90162} for $j=2$ and  
\eqref{eq:90312} shows 
\begin{align}
&\|W(t)\|_{H^1}
\|Q^{\nu}(t)\|_{H^{m+2}}
\leqslant 
C(T,\|Q_0\|_{H^m})
(\nu^{(m-1)-2}+\nu^{4-2}).
\label{eq:9017}
\end{align}
In addition, \eqref{eq:90162} for $j=4$ yields  
\begin{align}
(\nu^5-\mu^5)
\|Q^{\nu}(t)\|_{H^{m+4}}
&\leqslant 
C(T,\|Q_0\|_{H^m})
\nu.
\label{eq:90172}
\end{align}
Applying \eqref{eq:9017}-\eqref{eq:90172} to 
\eqref{eq:9016}, 
and then using \eqref{eq:H4Em} and $\nu\in (0,1)$, 
we derive 
\begin{align}
\frac{d}{dt}\|Z^m(t)\|_{L^2}^2
\leqslant 
C(T,\|Q_0\|_{H^m})
\left\{
\mathcal{E}^{\mu,\nu}_m(W(t))^2
+
(\nu^{m-3}+\nu)
\mathcal{E}^{\mu,\nu}_m(W(t))
\right\}. 
\nonumber
\end{align} 
The time-derivative of $\|W(t)\|_{H^{m-1}}^2$ 
can be also bounded by 
the right hand side of the above by a 
positive constant depending on $T$ and $\|Q_0\|_{H^m}$.  
Therefore 
\begin{align}
&\frac{d}{dt}
\mathcal{E}^{\mu,\nu}_m(W)^2
\leqslant 
C(T,\|Q_0\|_{H^m})
\left\{
\mathcal{E}^{\mu,\nu}_m(W)^2
+
(\nu^{m-3}+\nu)
\mathcal{E}^{\mu,\nu}_m(W)
\right\}
\nonumber 
\end{align} 
holds on $[0,T]$. 
By the Gronwall inequality 
and \eqref{eq:H4Em}, 
this shows   
$$
\|W\|_{C([0,T];H^m)}
\leqslant 
C(T,\|Q_0\|_{H^m})
\left(
\nu^{m-3}
 +\nu
+\|W(0)\|_{H^m}
\right),  
$$
which is the desired \eqref{eq:mcauchy}.
\end{proof}
\section{Proof of Theorem~\ref{theorem:lwp} and Corollary~\ref{cor:cor930}}
\label{section:prooflw}
This section completes the proof of Theorem~\ref{theorem:lwp} and 
Corollary~\ref{cor:cor930}.
\begin{proof}[Proof of Theorem~\ref{theorem:lwp}]
Let $Q_0\in H^m(\RR;\mathbb{C}^n)$ with an integer $m\geqslant 4$. 
From the time-reversibility of \eqref{eq:apde}, 
it suffices to solve \eqref{eq:apde}-\eqref{eq:adata} 
in positive time-direction. 
\par 
We show time-local existence of a solution. 
Let $\left\{Q_0^{\ep}\right\}_{\ep\in (0,1)}$ be the 
Bona-Smith approximation of $Q_0$ introduced in 
Section~\ref{section:local}. 
For $\mu$ and $\nu$ with $0<\mu\leqslant \nu<1$, 
let $Q^{\mu}$ and $Q^{\nu}$ be given in Section~\ref{section:BS}. 
Let $T=T(\|Q_0\|_{H^4})>0$ be given by \eqref{eq:extime} independently of $\mu$ and $\nu$.
Combining \eqref{eq:mcauchy} and the triangle inequality yields
\begin{align}
&\|Q^{\mu}-Q^{\nu}\|_{C([0,T];H^m)}
\leqslant 
C\left(
\nu^{m-3}+\nu
+\|Q_0^{\mu}-Q_0\|_{H^m}
+\|Q_0-Q_0^{\nu}\|_{H^m}
\right),
\nonumber
\end{align}
where $C>0$ is a constant depending on $T$ and $\|Q_0\|_{H^m}$ 
but not on $\mu$ and $\nu$.  
By the convergence $Q_0^{\ep}\to Q_0$ in $H^m$ 
as $\ep\downarrow 0$ 
and $m\geqslant 4$, this shows that 
$\left\{Q^{\mu}\right\}_{\mu\in (0,1)}$
is Cauchy in $C([0,T];H^m(\TT;\mathbb{C}^n))$, 
and thus there exists its limit $Q:=\displaystyle\lim_{\mu \downarrow 0}Q^{\mu}$ in $C([0,T];H^m(\TT;\mathbb{C}^n))$. 
It is not difficult to prove that $Q$  
actually satisfies \eqref{eq:apde}-\eqref{eq:adata}. 
\par 
We show uniqueness of the solution. 
Let $Q^1, Q^2\in C([0,T];H^4(\TT;\mathbb{C}^n))
$
be solutions to \eqref{eq:apde} with 
$Q^1(0,x)=Q^2(0,x)$.
Set 
$W:=Q^1-Q^2$.
We define 
\begin{align}
 &\mathcal{E}(W(t))
=
 \dfrac{1}{2}\|\p_xW(t)\|_{L^2}^2
 +A\|W(t)\|_{L^2}^2
 -\Re\langle
 W(t), M_a^{-1}\Lambda(Q^{1})W(t)
 \rangle, 
 \label{eq:p5313}
 \end{align}
 where 
 $\Lambda(Q^{1}) 
 =-\Lambda_1(Q^{1})+\Lambda_2(Q^{1})
 -\Lambda_3(Q^{1})$  
 and 
$\Lambda_k(Q^1)$ for $k=1,2,3$ are defined by \eqref{eq:u381}-\eqref{eq:u383} 
with $\mu$ replaced by $1$.  
 Moreover, $A=2C_1(\|Q^1\|_{C([0,T];H^4)})$ and 
 $C_1(\|Q^1\|_{C([0,T];H^4)})>1$ 
 is taken to be a sufficiently large constant which satisfies 
 \begin{align}
 \frac{1}{2}\|W(t)\|_{H^1}^2
 &\leqslant 
 \mathcal{E}(W(t))
 \leqslant 
 3 C_1(\|Q^1\|_{C([0,T];H^4)})
 \|W(t)\|_{H^1}^2
 \quad 
\text{on $[0,T]$.}
\label{eq:p92415}
\end{align} 
We can show \eqref{eq:92414}, \eqref{eq:9256} and 
\eqref{eq:9257} respectively with 
$(\mu,\nu,Q^{\mu},Q^{\nu})$ 
replaced by 
$(0,0,Q^{1},Q^{2})$.  
The computation to show them  
is formally the same as that 
we demonstrate in the previous section, 
which can be made rigorous since 
\begin{align}
 &\p_xW\in 
 C([0,T];H^3(\TT;\mathbb{C}^n))
 \cap C^1([0,T];H^{-1}(\TT;\mathbb{C}^n)). 
 \nonumber
 \end{align} 
Combining them and using \eqref{eq:p92415}, 
we obtain 
\begin{align}
\frac{d}{dt}
\mathcal{E}(W(t))
&\leqslant 
C(\|Q^{1}\|_{C([0,T];H^4)}, \|Q^{2}\|_{C([0,T];H^3)})
\mathcal{E}(W(t)). 
\nonumber
\end{align}
Since $Q^1(0,x)=Q^2(0,x)$, this combined with \eqref{eq:p92415}
shows 
$W=0$ on $[0,T]\times \TT$, which is the desired result.
\par 
Continuous dependence of the solution  
with respect to the initial data 
can be proved by using the standard techniques of Bona and Smith. 
More concretely, we can prove it by following, e.g., 
\cite[Section~4.3]{segata} and \cite[Section~7]{onodera6}, 
which is essentially based on 
the estimates in Proposition~\ref{proposition:cauchy} with $m\geqslant 4$. 
We omit the detail. 
\end{proof}
\begin{proof}[Proof of Corollary~\ref{cor:cor930}]
It suffices to show \eqref{eq:comm930}  
by using the additional (B7)-(B9) under (A1)-(A2) 
(without using 
$M_a=aI_n$ assumed in the proof of Theorem~\ref{theorem:lwp}.) 
\par 
We consider $\Lambda(Q^{\ep})=-\Lambda_1(Q^{\ep})+\Lambda_2(Q^{\ep})-\Lambda_3(Q^{\ep})$ 
defined by \eqref{eq:b381}-\eqref{eq:b383}.  
We can rewrite as follows:  
\begin{align}
&\Lambda_1(Q^{\ep})v=A_1(Q^{\ep})v+B_1(Q^{\ep})\overline{v}, 
\quad 
\Lambda_{\ell}(Q^{\ep})v=A_{\ell}(Q^{\ep})v \quad (\ell=2,3), 
\nonumber
\end{align}
for any $v=v(t,x): [0,T_{\ep}]\times \TT\to \mathbb{C}^n$ ,where 
\begin{align}
A_1(Q^{\ep})
&=
-\dfrac{1}{2}\left(\lambda^1_{jk}\right),  \ 
B_1(Q^{\ep})=\dfrac{1}{2}\left(\mu^1_{jk}\right), \ 
A_{\ell}(Q^{\ep})
=-\dfrac{1}{4}\left(\lambda^{\ell}_{jk}\right)
\ \  (\ell=2,3). 
\nonumber
\end{align}
Here $\left(\lambda^1_{jk}\right)$ denotes 
an $n\times n$
complex-matrix-valued function whose $(j,k)$-component is $\lambda^1_{jk}$, 
and the same applies to other 
$\left(\lambda^{\ell}_{jk}\right)$ 
and 
$\left(\mu^1_{jk}\right)$. 
A simple computation shows 
\begin{align}
\lambda^1_{jk}
&=
\sum_{q,r=1}^nS_{k,q,r}^{1,j}\overline{Q_q^{\ep}}Q_r^{\ep},\  
\mu_{jk}^1
=
\sum_{p,r=1}^nS_{p,k,r}^{1,j}Q_p^{\ep}Q_r^{\ep}, \ 
\lambda^2_{jk}
=
\sum_{p,q=1}^nS_{p,q,k}^{1,j}Q_p^{\ep}\overline{Q_q^{\ep}}, 
\nonumber
\\
\lambda^3_{jk}
&=
\sum_{p,q=1}^n
\left\{
S_{p,q,k}^{3,j}-(m-1)\left(
S_{p,q,k}^{1,j}-2S_{p,q,k}^{2,j}
\right)
\right\}Q_p^{\ep}\overline{Q_q^{\ep}}. 
\nonumber
\end{align}
Noting $\overline{M_av}=M_a\overline{v}$, we see 
\begin{align}
\left[
M_a,\Lambda(Q^{\ep})
\right]v
&=
-\left[
M_a, A_1(Q^{\ep})-A_2(Q^{\ep})+A_3(Q^{\ep})
\right]v
-
\left[
M_a, B_1(Q^{\ep})
\right]\overline{v}. 
\nonumber
\end{align}
Moreover, if we write 
$A_1(Q^{\ep})-A_2(Q^{\ep})+A_3(Q^{\ep})=
\left(\nu_{jk}\right)$, 
then 
\begin{align}
\nu_{jk}
&=
\dfrac{1}{4}
\sum_{p,q=1}^n
\left\{
S_{p,q,k}^{1,j}
-2S_{k,q,p}^{1,j}
-S_{p,q,k}^{3,j}+(m-1)\left(
S_{p,q,k}^{1,j}-2S_{p,q,k}^{2,j}
\right)
\right\}Q_p^{\ep}\overline{Q_q^{\ep}}.
\end{align}
This shows both $B_1(Q^{\ep})$ and 
$A_1(Q^{\ep})-A_2(Q^{\ep})+A_3(Q^{\ep})$ 
become diagonal and hence 
$\left[
M_a,\Lambda(Q^{\ep})
\right]=0$ holds, 
provided that all the following conditions are satisfied: 
\begin{align}
&S_{p,k,r}^{1,j}
=0
\quad 
\text{unless $j=k$} 
\quad 
\text{for all $p,r,j,k\in \left\{1,\ldots,n\right\}$}, 
\label{eq:add9301}
\\
&S_{p,q,k}^{1,j}-2S_{k,q,p}^{1,j}-S_{p,q,k}^{3,j}=0
\quad 
\text{unless $j=k$} 
\quad 
\text{for all $p,q,j,k\in \left\{1,\ldots,n\right\}$}, 
\label{eq:add9302}
\\
&S_{p,q,k}^{1,j}-2S_{p,q,k}^{2,j}=0
\quad 
\text{unless $j=k$} 
\quad 
\text{for all $p,q,j,k\in \left\{1,\ldots,n\right\}$}.  
\label{eq:add9303}
\end{align}
Here, by \eqref{eq:S1}, we see \eqref{eq:add9301} holds if and only if 
\begin{align}
&\omega_{k,p,r}^{2,j}
=0
\quad 
\text{unless $j=k$} 
\quad 
\text{for all $p,r,j,k\in \left\{1,\ldots,n\right\}$}, 
\nonumber 
\end{align}
which is nothing but (B8). 
Moreover, by \eqref{eq:S1} and \eqref{eq:S2}, 
$$
S_{p,q,k}^{1,j}-2S_{p,q,k}^{2,j}=
i\left(
\omega_{q,p,k}^{2,j}+\omega_{k,q,p}^{1,j}-\omega_{q,k,p}^{2,j}
\right)
=i\omega_{k,q,p}^{1,j}
$$
holds under (A1). This implies that \eqref{eq:add9303} 
is equivalent to (B7) under (A1). 
Furthermore, under (A1) and (A2), we use \eqref{eq:S1}-\eqref{eq:S3} 
to deduce 
\begin{align}
&S_{p,q,k}^{1,j}-2S_{k,q,p}^{1,j}-S_{p,q,k}^{3,j}
\nonumber
\\
&=
i\omega_{q,p,k}^{2,j}
-2i\omega_{q,k,p}^{2,j}
-i\left(
\omega_{k,q,p}^{1,j}-\omega_{q,k,p}^{2,j}
\right)
+\dfrac{i}{2}
\left(
2\omega_{k,q,p}^{1,j}+\omega_{k,q,p}^{3,j}+\omega_{k,p,q}^{4,j}+\omega_{p,k,q}^{4,j}
\right)
\nonumber
\\
&=
i\left(
\omega_{q,p,k}^{2,j}
-\omega_{q,k,p}^{2,j}
+\dfrac{1}{2}\omega_{k,q,p}^{3,j}+\dfrac{1}{2}\omega_{k,p,q}^{4,j}
+\dfrac{1}{2}\omega_{p,k,q}^{4,j}
\right)
\nonumber
\\
&=
i\left(
\dfrac{1}{2}\omega_{k,q,p}^{3,j}+\omega_{k,p,q}^{4,j}
\right).
\nonumber
\end{align}
This shows \eqref{eq:add9302} is equivalent to (B9) under (A1)-(A2). 
Therefore, we conclude that $[M_a,\Lambda(Q^{\ep})]=0$ follows 
from (B7)-(B9) under (A1)-(A2). 
\par 
Since $\Lambda(Q^{\mu})$ and $\Lambda(Q^{1})$ 
have essentially the same structure as that of $\Lambda(Q^{\ep})$, 
no other conditions are required to ensure   
$[M_a,\Lambda(Q^{\mu})]=0$ and $[M_a,\Lambda(Q^{1})]=0$.  
This completes the proof of Corollary~\ref{cor:cor930}. 
\end{proof}
\section{Example~\ref{ex:3}}
\label{section:examples}
Let $n=2$ and let $M_a=aI_2$ with $0\ne a\in \RR$ 
and $\lambda_1,\lambda_2\in \RR$.  
Then \eqref{eq:apde} becomes  
\begin{alignat}{2}
 \left(
 \p_t-i\, a\, \p_x^4-i\, \lambda_j \p_x^2
 \right)Q_j
 &=F_j(Q, \p_xQ, \p_x^2Q) 
 \quad 
 (j=1,2),  
\label{eq:21apde}
\end{alignat}
for 
$Q={}^t(Q_1,Q_2)(t,x):\RR\times \TT\to \mathbb{C}^2$, 
where $F_j(Q, \p_xQ, \p_x^2Q)$ for $j=1,2$ are given by 
\eqref{eq:b21} with $n=2$. 
We shall characterize  
$\left\{\omega_{p,q,r}^{k,j} \mid p,q,r,j\in \left\{1,2\right\}
  \right\}
$
for $k\in \left\{1,\ldots,4\right\}$ 
under \textup{(A1)-(A2)} 
and 
\textup{(B1)-(B6)}. 
By Proposition~\ref{proposition:BC}, 
it suffices to consider under 
\textup{(A1)-(A2)} 
and 
\textup{(C1)-(C4)}. 
We set $T_{p,q,r}^{k,j}:=if_k(p,q,r,j)$ and  
\begin{align}
T^k
&:=
\left(
\begin{array}{cc|cc} 
  T_{1,1,1}^{k,1} & T_{1,1,2}^{k,1} & T_{2,1,1}^{k,1}  & T_{2,1,2}^{k,1} \\ 
  T_{1,2,1}^{k,1} & T_{1,2,2}^{k,1} & T_{2,2,1}^{k,1}  & T_{2,2,2}^{k,1} \\ \hline
  T_{1,1,1}^{k,2} & T_{1,1,2}^{k,2} & T_{2,1,1}^{k,2}  & T_{2,1,2}^{k,2} \\ 
  T_{1,2,1}^{k,2} & T_{1,2,2}^{k,2} & T_{2,2,1}^{k,2}  & T_{2,2,2}^{k,2} 
\end{array}
\right)
\quad 
(k\in \left\{1,\ldots,4\right\}). 
\nonumber
\end{align}
\par 
Assume that all (C1)-(C4) are satisfied.  
Then $if_k(p,q,r,j)\in \Gamma_1\cap \Gamma_2$, 
and hence 
\begin{align}
T_{p,q,r}^{k,j}
&=if_k(p,q,r,j)=if_k(r,q,p,j)
=T_{r,q,p}^{k,j}
\quad 
\text{for all $p,q,r,j\in \left\{1,2\right\}$}, 
\label{eq:add9291}
\\
T_{p,q,r}^{k,j}
&=if_k(p,q,r,j)=\overline{if_k(j,r,q,p)}
=\overline{T_{j,r,q}^{k,p}}
\quad 
\text{for all $p,q,r,j\in \left\{1,2\right\}$} 
\label{eq:add9292}
\end{align}
for all $k\in \left\{1,\ldots,4\right\}$. 
This shows that   
the second column of $T^k$ coincides with the third one, 
and $T^k$ is Hermitian ($\overline{{}^{t}(T^k)}=T^k$). 
Moreover, \eqref{eq:add9291}-\eqref{eq:add9292} yields 
$T_{p,q,r}^{k,j}=\overline{T_{j,r,q}^{k,p}}
=\overline{T_{q,r,j}^{k,p}}=T_{p,j,r}^{k,q}$, and thus 
\begin{align}
T_{p,q,r}^{k,j}
&=T_{p,j,r}^{k,q}
\quad 
\text{for all $p,q,r,j\in \left\{1,2\right\}$}. 
\label{eq:add9294}
\end{align} 
This shows the second row of $T^k$ coincides with the third one. 
By \eqref{eq:add9291}-\eqref{eq:add9294}, 
we can express  
\begin{align}
T^k
&=
\left(
\begin{array}{cc|cc} 
  -\kappa_k & -\overline{\alpha_k} & -\overline{\alpha_k}  & -\overline{\gamma_k} \\ 
  -\alpha_k & -\tau_k & -\tau_k  & -\overline{\beta_k} \\ \hline
  -\alpha_k & -\tau_k & -\tau_k  &  -\overline{\beta_k} \\ 
  -\gamma_k & -\beta_k & -\beta_k  & -\sigma_k 
\end{array}
\right)
\quad 
\text{($k\in \left\{1,\ldots,4\right\}$)}
\label{eq:Tk}
\end{align}
by using independent constants 
$\alpha_k, \beta_k, \gamma_k\in \mathbb{C}$ 
and $\kappa_k,\tau_k,\sigma_k\in \RR$.  
\par 
Recalling 
$\omega_{p,q,r}^{1,j}=-iT_{p,q,r}^{1,j}$,  
$\omega_{q,p,r}^{2,j}=-iT_{p,q,r}^{2,j}$, 
$\omega_{p,q,r}^{3,j}=-iT_{p,q,r}^{3,j}$,  
$\omega_{p,r,q}^{4,j}=-iT_{p,q,r}^{4,j}$
follow from the definition of $T_{p,q,r}^{k,j}$, 
and using \eqref{eq:Tk}, we derive 
\begin{align}
&\left(
\begin{array}{cc|cc} 
  \omega_{1,1,1}^{k,1} & \omega_{1,1,2}^{k,1} & \omega_{2,1,1}^{k,1}  & \omega_{2,1,2}^{k,1} \\ 
  \omega_{1,2,1}^{k,1} & \omega_{1,2,2}^{k,1} & \omega_{2,2,1}^{k,1}  & \omega_{2,2,2}^{k,1} \\ \hline
  \omega_{1,1,1}^{k,2} & \omega_{1,1,2}^{k,2} & \omega_{2,1,1}^{k,2}  & \omega_{2,1,2}^{k,2} \\ 
  \omega_{1,2,1}^{k,2} & \omega_{1,2,2}^{k,2} & \omega_{2,2,1}^{k,2}  & \omega_{2,2,2}^{k,2} 
\end{array}
\right) 
=
-i
\left(
\begin{array}{cc|cc} 
  -\kappa_k & -\overline{\alpha_k} & -\overline{\alpha_k}  & -\overline{\gamma_k} \\ 
  -\alpha_k & -\tau_k & -\tau_k  & -\overline{\beta_k} \\ \hline
  -\alpha_k & -\tau_k & -\tau_k  &  -\overline{\beta_k} \\ 
  -\gamma_k & -\beta_k & -\beta_k  & -\sigma_k 
\end{array}
\right) 
\quad 
(k=1,3),
\label{eq:b3511}
\\
&\left(
\begin{array}{cc|cc} 
  \omega_{1,1,1}^{2,1} & \omega_{1,1,2}^{2,1} & \omega_{1,2,1}^{2,1}  & \omega_{1,2,2}^{2,1} \\ 
  \omega_{2,1,1}^{2,1} & \omega_{2,1,2}^{2,1} & \omega_{2,2,1}^{2,1}  & \omega_{2,2,2}^{2,1} \\ \hline
  \omega_{1,1,1}^{2,2} & \omega_{1,1,2}^{2,2} & \omega_{1,2,1}^{2,2}  & \omega_{1,2,2}^{2,2} \\ 
  \omega_{2,1,1}^{2,2} & \omega_{2,1,2}^{2,2} & \omega_{2,2,1}^{2,2}  & \omega_{2,2,2}^{2,2} 
\end{array}
\right) 
=
-i
\left(
\begin{array}{cc|cc} 
  -\kappa_2 & -\overline{\alpha_2} & -\overline{\alpha_2}  & -\overline{\gamma_2} \\ 
  -\alpha_2 & -\tau_2 & -\tau_2  & -\overline{\beta_2} \\ \hline
  -\alpha_2 & -\tau_2 & -\tau_2  &  -\overline{\beta_2} \\ 
  -\gamma_2 & -\beta_2 & -\beta_2  & -\sigma_2 
\end{array}
\right), 
\label{eq:b3512}
\\
&\left(
\begin{array}{cc|cc} 
  \omega_{1,1,1}^{4,1} & \omega_{1,2,1}^{4,1} & \omega_{2,1,1}^{4,1}  & \omega_{2,2,1}^{4,1} \\ 
  \omega_{1,1,2}^{4,1} & \omega_{1,2,2}^{4,1} & \omega_{2,1,2}^{4,1}  & \omega_{2,2,2}^{4,1} \\ \hline
  \omega_{1,1,1}^{4,2} & \omega_{1,2,1}^{4,2} & \omega_{2,1,1}^{4,2}  & \omega_{2,2,1}^{4,2} \\ 
  \omega_{1,1,2}^{4,2} & \omega_{1,2,2}^{4,2} & \omega_{2,1,2}^{4,2}  & \omega_{2,2,2}^{4,2} 
\end{array}
\right) 
=
-i
\left(
\begin{array}{cc|cc} 
  -\kappa_4 & -\overline{\alpha_4} & -\overline{\alpha_4}  & -\overline{\gamma_4} \\ 
  -\alpha_4 & -\tau_4 & -\tau_4  & -\overline{\beta_4} \\ \hline
  -\alpha_4 & -\tau_4 & -\tau_4  &  -\overline{\beta_4} \\ 
  -\gamma_4 & -\beta_4 & -\beta_4  & -\sigma_4 
\end{array}
\right). 
\label{eq:b3521}
\end{align}
\par 
Conversely, if
$\left\{\omega_{p,q,r}^{k,j} \mid p,q,r,j\in \left\{1,2\right\}
  \right\}
$
for $k\in \left\{1,\ldots,4\right\}$ are given by
\eqref{eq:b3511}-\eqref{eq:b3521}, 
then all the conditions 
(C1)-(C4) and (A1)-(A2) are obviously satisfied. 
\begin{remark}
\label{remark:add930}
The additional  \textup{(B7)-(B9)} 
imposed in Corollary~\ref{cor:cor930} to discuss the case $M_a\ne aI_n$
seems to be rather strong. 
Indeed, only by imposing \textup{(B7)-(B8)}, 
\eqref{eq:b3511} for $k=1$ and \eqref{eq:b3512} 
respectively reduce to 
\begin{align}
&\left(
\begin{array}{cc|cc} 
  \omega_{1,1,1}^{1,1} & \omega_{1,1,2}^{1,1} & \omega_{2,1,1}^{1,1}  & \omega_{2,1,2}^{1,1} \\ 
  \omega_{1,2,1}^{1,1} & \omega_{1,2,2}^{1,1} & \omega_{2,2,1}^{1,1}  & \omega_{2,2,2}^{1,1} \\ \hline
  \omega_{1,1,1}^{1,2} & \omega_{1,1,2}^{1,2} & \omega_{2,1,1}^{1,2}  & \omega_{2,1,2}^{1,2} \\ 
  \omega_{1,2,1}^{1,2} & \omega_{1,2,2}^{1,2} & \omega_{2,2,1}^{1,2}  & \omega_{2,2,2}^{1,2} 
\end{array}
\right) 
=
-i
\left(
\begin{array}{cc|cc} 
  -\kappa_1 & 0 & 0  & 0 \\ 
  0 & 0 & 0  & 0 \\ \hline
  0 & 0 & 0  &  0 \\ 
  0 & 0 & 0  & -\sigma_1 
\end{array}
\right), 
\nonumber
\\
&\left(
\begin{array}{cc|cc} 
  \omega_{1,1,1}^{2,1} & \omega_{1,1,2}^{2,1} & \omega_{1,2,1}^{2,1}  & \omega_{1,2,2}^{2,1} \\ 
  \omega_{2,1,1}^{2,1} & \omega_{2,1,2}^{2,1} & \omega_{2,2,1}^{2,1}  & \omega_{2,2,2}^{2,1} \\ \hline
  \omega_{1,1,1}^{2,2} & \omega_{1,1,2}^{2,2} & \omega_{1,2,1}^{2,2}  & \omega_{1,2,2}^{2,2} \\ 
  \omega_{2,1,1}^{2,2} & \omega_{2,1,2}^{2,2} & \omega_{2,2,1}^{2,2}  & \omega_{2,2,2}^{2,2} 
\end{array}
\right) 
=
-i
\left(
\begin{array}{cc|cc} 
  -\kappa_2 & 0 & 0  & -\overline{\gamma_2} \\ 
  0 & 0 & 0  & 0 \\ \hline
  0 & 0 & 0  &  0 \\ 
  -\gamma_2 & 0 & 0  & -\sigma_2 
\end{array}
\right).
\nonumber
\end{align}
\end{remark}
%
%
%
\section{Relevance to a geometric dispersive PDE}
\label{section:GB}
We state the background of our results with relevance to the study on 
\eqref{eq:pde} 
for curve flows on a compact locally Hermitian symmetric space $(N,J,h)$ of complex dimension $n$.  
\footnote{
The definition of $R$ is adopted here to be   
$
R(V_1,V_2)V_3:=\nabla_{V_1}\nabla_{V_2}V_3
-\nabla_{V_2}\nabla_{V_1}V_3-\nabla_{[V_1,V_2]}V_3
$ 
for any vector fields $V_1, V_2,V_3$ on $N$ where $[V_1,V_2]:=V_1V_2-V_2V_1$.
}
Geometrically,  \eqref{eq:pde} describes
the relationship among elements of $\Gamma(u^{-1}TN)$, 
where $\Gamma(u^{-1}TN)$ denotes the set of 
sections of the pull-back bundle $u^{-1}TN$ induced by $u$. 
For the details on the origin of \eqref{eq:pde}, 
see \cite{onodera4,onodera5,onoderamomo} and references therein. 
\subsection{PDE satisfied by $\nabla_x^mu_x$.}
\label{subsection:GB1}
The initial value problem 
for \eqref{eq:pde} on $\TT$  
was investigated in \cite{onodera4,onodera5}.  
In particular, \cite{onodera4} ensures   
it possesses a time-local solution for any initial data $u_0$ with   
$u_{0x}$ being in a geometric Sobolev space $H^m(\TT;TN)$ with integer
$m\geqslant 4$, 
where observation of the equation 
satisfied by $\nabla_x^mu_x$ played a crucial role.  
We review it briefly: 
Let $u:[-T,T]\times \TT\to N$ be a sufficient smooth solution to 
\eqref{eq:pde}.    
Then, a basic computation in K\"ahler geometry
shows 
\begin{align}
&\left(\nabla_t-
a\,J_u\nabla_x^4
-\lambda\,J_u\nabla_x^2
\right)
\nabla_x^mu_x
\nonumber
\\
&=
d_{1,m}\,R(\nabla_x^2\nabla_x^mu_x,J_uu_x)u_x
+d_{2,m}\,\nabla_x\left\{
R(J_uu_x,u_x)\nabla_x\nabla_x^mu_x
\right\}
\nonumber
\\
&\quad
+d_{3,m}\,R(J_u\nabla_xu_x, u_x)\nabla_x\nabla_x^mu_x
+d_{4,m}\,R(\nabla_xu_x,u_x)J_u\nabla_x\nabla_x^mu_x
\nonumber
\\
&\quad
+d_{5,m}\left\{
R(\nabla_x\nabla_x^mu_x,\nabla_xu_x)J_uu_x+
R(\nabla_x\nabla_x^mu_x,J_uu_x)\nabla_xu_x
\right\}
+r_m 
\label{eq:aya02}
\end{align}
where $d_{1,m}, \ldots,d_{5,m}$ are real constants depending on $m$, 
and $r_m=r_m(t,x)$ satisfies 
$$\|r_m(t)\|_{L^2(\TT;TN)}\leqslant C(\|u_x(t)\|_{H^3(\TT;TN)})\|u_x(t)\|_{H^m(\TT;TN)}
\quad 
\text{for $t\in [-T,T]$.}
$$ 
The classical energy estimate 
for 
$$\|u_x(t)\|_{H^m(\TT;TN)}^2
=\sum_{k=0}^m\|\nabla_x^ku_x(t)\|_{L^2(\TT;TN)}^2
=\sum_{k=0}^m\int_{0}^{2\pi}h\left(
\nabla_x^ku_x(t),\nabla_x^ku_x(t)
\right)dx
$$ 
involves 
loss of derivatives that comes only from the terms 
$R(\nabla_x^2\nabla_x^mu_x,J_uu_x)u_x$  
and 
$R(J_u\nabla_xu_x, u_x)\nabla_x\nabla_x^mu_x$ 
in the right hand side of \eqref{eq:aya02}. 
The difficulty was overcome 
by estimating $\|u_x\|_{H^{m-1}(\TT;TN)}^2+\|V_m(t)\|_{L^2(\TT;TN)}^2$
in place of  $\|u_x(t)\|_{H^m(\TT;TN)}^2$, 
where 
\begin{align}
V_m&=
\nabla_x^mu_x+\dfrac{d_{1,m}}{2a}R(\nabla_x^{m-2}u_x,u_x)u_x
+\dfrac{d_{1,m}-d_{3,m}}{8a}R(J_uu_x,u_x)J_u\nabla_x^{m-2}u_x. 
\label{eq:gaya02}
\end{align}
The transformation $\nabla_x^mu_x\mapsto V_m$ behaves as 
a summation of the identity and a pseudodifferential operator of order $-2$, and the commutator with $aJ_u\nabla_x^4$  
cancels out the loss of derivatives.
Additionally,  
$\nabla R=0$ and the K\"ahlarity $\nabla J=0$  
that characterize $(N,J,h)$ as a locally Hermitian symmetric space 
prevents  
other terms with loss of derivatives and 
the worst term of the form 
$(\nabla R)(u_x)(J_uu_x,u_x)\nabla_x\nabla_x^mu_x$ from appearing in the right hand 
side of \eqref{eq:aya02}. 
\subsection{Reduction to a system of PDEs.}
\label{subsection:GB2}
We investigate the structure of \eqref{eq:aya02} 
in the level of an $n$-component system  
for complex valued functions, 
by using the Generalized Hasimoto transformation.  
The method has been developed by several authors 
to clarify essential structures of some geometric 
dispersive partial differential equations 
for curve flows on K\"ahler manifolds. 
See, e.g., 
\cite{CSU,chihara2015,DZ2021,Koiso1997,
NSVZ,onodera0,onoderamomo,
RRS,SW2011,SW2013} and references therein.
We here follow the notation and 
the framework in \cite[Section~2]{onoderamomo} 
which were mainly motivated by \cite{CSU,RRS}. 
We should note that the investigation below is nothing but an 
observation in that we do not handle the periodic case $X=\TT$. 
(See Remark~\ref{remark:hasimoto}.) 
\par 
Let $u^{\infty}$ be a fixed point on $N$, 
and let 
$u=u(t,x):[-T,T]\times \RR\to N$ be a smooth solution 
to \eqref{eq:pde} such that 
$\displaystyle\lim_{x\to -\infty}u(t,x)=u^{\infty}$ 
and $u_x(t,\cdot):\RR\to (u(t,\cdot))^{-1}TN$ is in the Schwartz class for any $t\in [-T,T]$.  
Let $\left\{e^{\infty}_1, \ldots, e^{\infty}_n, 
J_{u^{\infty}}e^{\infty}_{1},\ldots, J_{u^{\infty}}e^{\infty}_{n}\right\}$ 
be an orthonormal basis
for $T_{u^{\infty}}N$ with respect to $h$. 
Following \cite{CSU} and \cite{RRS}, 
we take the orthonormal frame 
$\left\{e_1, \ldots, e_n, e_{n+1},\ldots, e_{2n}\right\}$
for $u^{-1}TN$ that satisfies 
\begin{align}
\nabla_xe_p(t,x)&=0,  
\quad 
\lim_{x\to -\infty}e_p(t,x)=e_p^{\infty}, 
\label{eq:mf01} 
\end{align} 
and $e_{p+n}=J_ue_p$ for any $p\in \left\{1,\ldots,n\right\}$. 
The K\"ahlerity $\nabla J=0$ ensures 
$\nabla_xe_{p+n}=0$ for any $p\in \left\{1,\ldots,n\right\}$. 
In the setting, we can write 
 \begin{align}
 \nabla_te_p
 &=
 \sum_{j=1}^{n}a_j^pe_j
 +
 \sum_{j=1}^{n}b_j^pJ_ue_j
 \quad 
 (\forall p\in \left\{1,\ldots,n\right\}), 
 \label{eq:tep}
 \end{align}
where 
$a_j^p=a_j^p(t,x)$ and 
$b_j^p=b_j^p(t,x)$ 
denote real-valued functions of $(t,x)$.    
We define   
\begin{align}
S_{p,q,r}^j
&=
\dfrac{1}{2}
\left\{
h(R(e_p,e_q)e_r,e_j)
+i\,h(R(e_p,e_q)e_r,J_ue_j)
\right\}
\nonumber
\\
&\quad
+
\dfrac{i}{2}
\left\{
h(R(e_p,J_ue_q)e_r,e_j)
+i\,h(R(e_p,J_ue_q)e_r,J_ue_j)
\right\}
\nonumber 
\end{align}
for $p,q,r,j\in \left\{1,\ldots,n\right\}$. 
The basic properties of $R$ are reflected to $S_{p,q,r}^j$ as follows: 
\begin{proposition}[\cite{onoderamomo}, Proposition~2.6]
\label{prop:tensor}
For any $U,V,W\in \Gamma(u^{-1}TN)$, 
\begin{align}
\langle
R(U,V)W
\rangle_j
&=
\sum_{p,q,r=1}^n
	  S_{p,q,r}^{j}\left(
	  \langle U \rangle_p\overline{\langle V \rangle_q}
	 -\langle V \rangle_p\overline{\langle U \rangle_q}
	  \right)
	  \langle W \rangle_r
 \quad 
 (\forall j\in \left\{1,\ldots,n\right\}).	  
	  \label{eq:23613}	 
\end{align}
\end{proposition}
Here, for any $\Xi\in \Gamma(u^{-1}TN)$, 
$\langle\Xi\rangle_j$ 
denotes a complex-valued function defined by 
$\langle\Xi\rangle_j=h(\Xi,e_j)+i\,h(\Xi,J_ue_j)$, 
which equals to 
\begin{align}
\langle\Xi\rangle_j
&=
\Xi_j^R+i\,\Xi_j^I,  
\label{eq:momo2}
\end{align}
if we write  
$\Xi=\displaystyle\sum_{k=1}^n
   (\Xi_k^R+J_u\Xi_k^I)e_k$ 
by using real-valued functions 
$\Xi_k^R$ and $\Xi_k^I$. 
\begin{proposition}[\cite{onoderamomo}, Proposition~2.7]
\label{proposition:R4}
\begin{align}
S_{p,q,r}^j
&=
S_{r,q,p}^j
\quad
(\forall p,q,r,j\in \left\{1,\ldots,n\right\}). 
\label{eq:tsu3}
\end{align}
\end{proposition}
\begin{proposition}
\label{proposition:R3}
\begin{align}
S_{p,q,r}^j
&=
\overline{
S_{j,r,q}^p
}
\quad
(\forall p,q,r,j\in \left\{1,\ldots,n\right\}). 
\nonumber 
\end{align}
\end{proposition}
\begin{proposition}[\cite{onoderamomo}, Proposition ~2.3]
\label{proposition:Rloc} 
\begin{align}
&\p_x S_{p,q,r}^{j}=0
\quad
(\forall p,q,r,j\in \left\{1,\ldots,n\right\}).
\nonumber
\end{align}
\end{proposition}
Although Proposition~\ref{proposition:R3} has not been 
presented in \cite{onoderamomo}, 
we can easily show it by using the following basic properties 
of $(N,J,h)$ as a K\"ahler manifold: 
\begin{align}
&
h(R(Y_1,Y_2)Y_3,Y_4)
=h(R(Y_3,Y_4)Y_1,Y_2)
=h(R(Y_4,Y_3)Y_2,Y_1), 
\nonumber
\\
&R(J_uY_1,Y_2)Y_3=-R(Y_1,J_uY_2)Y_3
\nonumber
\end{align}
for any $Y_1,\ldots,Y_4\in \Gamma(u^{-1}TN)$. 
We omit the detail.  
\par 
We set $U_m=\nabla_x^mu_x$ and 
represent $u_x, u_t, U_m\in \Gamma(u^{-1}TN)$ by 
\begin{align}
u_x&=\sum_{p=1}^{n}(\xi_p+\eta_pJ_u)e_p, 
u_t=\sum_{p=1}^{n}(\mu_p+\nu_pJ_u)e_p,  
U_m=\sum_{p=1}^{n}(V_p+J_uW_p)e_p, 
 \label{eq:Um}
 \end{align} 
where $\xi_p$, $\eta_p$, $\mu_p$, $\nu_p$, $V_p$, $W_p$ 
are real-valued functions of $(t,x)$. 
Set 
$Q_j:=\langle u_x\rangle_j$, 
$P_j:=\langle u_t\rangle_j$,  
$Z_j:=\langle U_m\rangle_j$
for 
$j\in \left\{1,\ldots,n\right\}$. 
By \eqref{eq:momo2}, they satisfy  
 \begin{align}
 Q_j&=\xi_j+i\eta_j, 
 \quad
 P_j=\mu_j+i\nu_j, 
 \quad 
 Z_j=V_j+iW_j.
\nonumber 
 \end{align}  
It follows from \eqref{eq:Um} and \eqref{eq:tep}  
\begin{align}
 \nabla_tU_m
 &=
 \sum_{p=1}^{n}(\p_tV_p+J_u\p_tW_p)e_p
 +
 \sum_{p=1}^{n}(V_p+J_uW_p)\nabla_t e_p. 
 \nonumber
\end{align}
Since $J_u^2=-\textrm{id}$ being the minus identity on $\Gamma(u^{-1}TN)$, this shows 
\begin{align}
\langle
\nabla_tU_m
\rangle_j
   &= 
   \p_tV_j+i\p_tW_j
   +\sum_{p=1}^{n}\left\{
   (a_j^pV_p-b_j^pW_p)+i(a_j^pW_p+b_j^pV_p)
    \right\}
   \nonumber
   \\
   &=
      \p_tZ_j
      +\sum_{p=1}^{n}
      (a_j^p+ib_j^p)Z_p.
      \label{eq:41}
 \end{align}
 Here, as is proved 
   in \cite[Eq. (2.61)]{onoderamomo},  
  \begin{align}
  \sum_{p=1}^{n}
  (a_j^p+ib_j^p)Z_p
  &=
  O\left(
  (\|Q\|_{C([-T,T];H^3(\RR;TN))}^2
  +\|Q\|_{C([-T,T];H^3(\RR;TN))}^4)
  |Z|
  \right).
  \label{eq:161}
  \end{align}
From 
$J_u^2=-\textrm{id}$, 
$\nabla J=0$, 
\eqref{eq:mf01}, 
and \eqref{eq:momo2}, 
we have   
\begin{align}
\langle
\left(
 a\,J_u\nabla_x^4
 +\lambda\,J_u\nabla_x^2
 \right)
 U_m
\rangle_j
  &=
   i (a\p_x^4+\lambda \p_x^2)Z_j. 
   \label{eq:42}
 \end{align} 
Further, applying \eqref{eq:23613} in Proposition~\ref{prop:tensor}, 
we can obtain  
\begin{align}
\langle 
R(\nabla_x^2U_m,J_uu_x)u_x
\rangle_j
&=
-i
\sum_{p,q,r=1}^n
	  S_{p,q,r}^{j}\left(
	  \p_x^2Z_p
	  \overline{Q_q}
	 +Q_p\overline{\p_x^2Z_q}
	  \right)
	  Q_r, 
	  \label{eq:91}
\\
\langle 
\nabla_x\left\{
R(J_uu_x,u_x)\nabla_xU_m
\right\}
\rangle_j
&=
\p_x\left\{
2i\sum_{p,q,r=1}^n
	  S_{p,q,r}^{j}
	  Q_p
	  \overline{Q_q}
	  \p_xZ_r
\right\}, 
\label{eq:103}
\\
\langle 
R(J_u\nabla_xu_x,u_x)\nabla_xU_m
\rangle_j
&=
i\sum_{p,q,r=1}^n
	  S_{p,q,r}^{j}
	  \p_x(Q_p\overline{Q_q})
	  \p_xZ_r, 
\label{eq:111}
\\
\langle 
R(\nabla_xu_x,u_x)J_u\nabla_xU_m
\rangle_j
&=i\sum_{p,q,r=1}^n
	  S_{p,q,r}^{j}\left(
	  \p_xQ_p\overline{Q_q}
	 -Q_p\overline{\p_xQ_q}
	  \right)
	  \p_xZ_r, 
	  \label{eq:121}
\end{align}
and 
\begin{align}
&
\langle 
R(\nabla_xU_m,\nabla_xu_x)J_uu_x
+
R(\nabla_xU_m,J_uu_x)\nabla_xu_x
\rangle_j
\nonumber
\\
&=-i\sum_{p,q,r=1}^n
	  	  	  S_{p,q,r}^{j}
	  	  	   \left(
	  	  	 \p_xQ_p\overline{Q_q}
	  	  	 -Q_p\overline{\p_xQ_q}
	  	  	  \right)\p_xZ_r 
	  	  -2i\sum_{p,q,r=1}^n
	  	  	  	  S_{p,q,r}^{j}
	  	  	  	  \overline{\p_xZ_q}
	  	  	  	  \p_xQ_pQ_r.	  		  	   
	  	  \label{eq:141}
\end{align} 
We show only \eqref{eq:141} here as an example. 
Applying \eqref{eq:23613} 
for 
$U=\nabla_xU_m$, $V=\nabla_xu_x$, $W=J_uu_x$
where  $\langle\nabla_xU_m\rangle_p=\p_xV_p+i\p_xW_p
=\p_xZ_p$, 
$\langle \nabla_xu_x\rangle_q=\p_x\xi_q+i\p_x\eta_q=\p_xQ_q$, 
$\langle J_uu_x\rangle_r=-\eta_r+i\xi_r=iQ_r$, 
we obtain 
\begin{align}
L_1&:=\langle 
R(\nabla_xU_m,\nabla_xu_x)J_uu_x
\rangle_j
=
i\sum_{p,q,r=1}^n
	  S_{p,q,r}^{j}\left(
	  \p_xZ_p\overline{\p_xQ_q}
	 -\p_xQ_p\overline{\p_xZ_q}
	  \right)
	  Q_r. 
\nonumber
\end{align}
In the same way, applying \eqref{eq:23613} for 
$U=\nabla_xU_m$, $V=J_uu_x$, $W=\nabla_xu_x$
where  $\langle\nabla_xU_m\rangle_p=\p_xZ_p$, 
$\langle J_uu_x\rangle_q=iQ_q$, 
$\langle \nabla_xu_x\rangle_r=\p_xQ_r$, 
we obtain  
\begin{align}
L_2&:=\langle
R(\nabla_xU_m,J_uu_x)\nabla_xu_x
\rangle_j
=
-i\sum_{p,q,r=1}^n
	  	  S_{p,q,r}^{j}\left(
	  	  \p_xZ_p\overline{Q_q}
	  	 +Q_p\overline{\p_xZ_q}
	  	  \right)
	  	  \p_xQ_r. 
\nonumber
\end{align}
From them, we have 
\begin{align}
&L_1+L_2
=
-i\sum_{p,q,r=1}^n
	  S_{p,q,r}^{j}
	   \p_xZ_p
	   \left(
	   \overline{Q_q}\p_xQ_r-
	 \overline{\p_xQ_q}Q_r
	  \right)
\nonumber
\\
&\quad\qquad\qquad
-i\sum_{p,q,r=1}^n
	  	  S_{p,q,r}^{j}
	  	  \overline{\p_xZ_q}
	  	  \left(
	  	  \p_xQ_pQ_r
	  	  +Q_p\p_xQ_r
	  	  \right).  
\nonumber
\end{align} 
Further, by replacing indexes $p$ and $r$ in both summation terms in 
the right hand side of the above, and by using 
\eqref{eq:tsu3} in Proposition~\ref{proposition:R4}, 
we see $L_1+L_2$ can be rewritten as the right hand side of 
\eqref{eq:141}. 
\par
Consequently, substituting \eqref{eq:aya02} into 
\eqref{eq:41} and combining 
\eqref{eq:161}-\eqref{eq:141}, we see  
$Z={}^t(Z_1,\ldots,Z_n)(t,x):[-T,T]\times \RR\to \mathbb{C}^n$ 
satisfies
\begin{align}
&(\p_t-iM_a\p_x^4-iM_{\lambda} \p_x^2)Z
\nonumber
\\
&=
d_{1,m}P_1^1(Q)\p_x^2Z
+d_{2,m}\p_x\left\{
P_2^2(Q)\p_xZ
\right\}
+d_{3,m}P_3^3(Q)\p_xZ
\nonumber
\\
&\quad
+(d_{4,m}-d_{5,m})P_4^4(Q)\p_xZ
+d_{5,m}P_5^5(Q)\p_xZ
+s_{m}, 
	  \label{eq:361}
	  \end{align}
where $M_a=aI_n$, $M_{\lambda}=\lambda I_n$, 
$P_{k}^k(Q)$ are defined by \eqref{eq:b511}-\eqref{eq:b515} 
under the special setting $S_{p,q,r}^{1,j}=\ldots=S_{p,q,r}^{5,j}=S_{p,q,r}^j$ 
for all $p,q,r,j\in \left\{1,\ldots,n\right\}$, and 
$\left\|s_{m}(t)\right\|_{L^2(\RR)}\leqslant 
C(\|Q\|_{C([-T,T];H^3(\RR))})\|Q(t)\|_{H^m}$. 
\par 
The bad terms 
$R(\nabla_x^2U_m, J_uu_x)u_x$ and 
$R(J_u\nabla_xu_x,u_x)\nabla_xU_m$ 
in \eqref{eq:aya02} 
are expressed by 
\eqref{eq:91} and 
\eqref{eq:111} respectively. 
Hence, it is expected that 
the classical energy estimate for the time-derivative 
of $\|Z(t)\|_{L^2(\RR)}^2$ (to be bounded above by $\|Q(t)\|_{H^m(\RR)}^2$ 
multiplied by a positive constant)
involves loss of derivatives which comes  from 
the first and the third terms of the right hand side of  
\eqref{eq:361}, and no loss of derivatives occurs from the other terms thanks to 
the properties of $S_{p,q,r}^j$ in 
Propositions~\ref{proposition:R4} 
and \ref{proposition:R3}.  
Based on the expectation,  
the author decided to investigate  
$\p_x\left(F(Q,\p_xQ,\p_x^2Q)\right)$
and $\p_x^m\left(F(Q,\p_xQ,\p_x^2Q)\right)$ 	 
for \eqref{eq:apde} by using $P_k^{\ell}(Q)$ 
defined by \eqref{eq:b511}-\eqref{eq:b515} 
with the conditions (G1)-(G5), 
finding the equivalent ones (B1)-(B6) which are 
easier to be checked whether they are satisfied or not.  
Additionally, the function $V_m$ defined by \eqref{eq:gaya02} 
turns out to satisfy 
\begin{align} 
\langle 
V_m
\rangle_j
&=
Z_j+
\dfrac{d_{1,m}}{2a}
\sum_{p,q,r=1}^n
S_{p,q,r}^j
\left(
\p_x^{m-2}Q_p\overline{Q_q}
-
Q_p\overline{\p_x^{m-2}Q_q}
\right)
Q_r
\nonumber 
\\
&\quad 
-
\dfrac{d_{1,m}-d_{3,m}}{4a}
\sum_{p,q,r=1}^n
S_{p,q,r}^j
Q_p\overline{Q_q}
\p_x^{m-2}Q_r
\nonumber
\end{align}
for each $j\in \left\{1,\ldots,n\right\}$. 
This gave the author a hint to arrive at the form \eqref{eq:V_j} with 
\eqref{eq:b381}-\eqref{eq:b383} and the analogical ones   
to prove Theorem~\ref{theorem:lwp}.  
\begin{remark}
\label{remark:hasimoto} 
The above transformation makes sense 
if $X=\RR$
and the solution $u$ has an edge point on $N$ as $x\to-\infty$. 
Some additional arguments are required 
to apply the transformation rigorously if $X=\TT$, 
involving the holonomy correction associated with the closed curve solution, 
which is indicated by the previous study \cite{chihara2015,RRS}. 
Fortunately however, without pursuing the direction, 
we can derive a sufficient information to arrive at \textup{(B1)-(B6)} 
only from the observation under $X=\RR$. 
This may be mainly because the loss of derivatives does not occur from the 
non-local integral terms of the system (\cite[Eq. (1.10)]{onoderamomo})
which destroy the periodicity in $x$. 
\end{remark}
%
\section*{Acknowledgments}
The author has been supported by 
JSPS Grant-in-Aid for Scientific Research (C) 
Grant Numbers JP20K03703 and JP24K06813. 


\begin{thebibliography}{00}

\bibitem{BS}
Bona,~J.~L., Smith,~R.:  
The initial value problem for the Korteweg-de Vries equation.  
 Philos. Trans. R. Soc. Lond., Ser. A  
\textbf{278}, 555--601 (1975)

\bibitem{CSU}
Cahng,~N.-H., Shatah,~J., Uhlenbeck,~K.: 
Schr\"odinger maps. 
Comm.\ Pure Appl.\ Math.  
\textbf{53}, 590--602 (2000) 
 
\bibitem{chihara2015} 
Chihara,~H.:
Fourth-order dispersive systems on the one-dimensional torus.  
J.\ Pseudo-Differ.\ Oper.\ Appl. \textbf{6}, 237--263 (2015)
 
 \bibitem{DKA}
 Daniel,~M., Kavitha,~L., Amuda,~R.:  
Soliton spin excitations in an anisotropic Heisenberg ferromagnet 
 with octupole-dipole interaction. 
 Phys.\ Rev.\ B \textbf{59}, 13774 (1999)
 
 \bibitem{DL}
Daniel,~M., Latha,~ M.~M.:  
Soliton in discrete and continuum alpha helical proteins with higher-order excitations.  
 Phys. A \textbf{240}, 526--546 (1997) 

\bibitem{DW2018}
Ding,~Q., Wang,~Y.~D.:  
Vortex filament on symmetric Lie algebras and 
generalized bi-Schr\"odinger flows.  
Math.\ Z. \textbf{290}, 167--193 (2018) 

\bibitem{DZ2021}
Ding,~Q., Zhong,~S.:  
On the vortex filament in 3-spaces and its generalizations.   
Sci. China Math.
\textbf{64}, 1331--1348 (2021) 

\bibitem{ET}
 Erdogan,~M.~B., Tzirakis,~N.: 
Dispersive Partial Differential Equations; Wellposedness and Applications.  
 Cambridge Student Texts, 86, Cambridge University Press (2016)
 
 \bibitem{fukumoto}
 Fukumoto,~Y.: 
Three-dimensional motion of a vortex filament and its relation 
 to the localized induction hierarchy, 
Eur.\ Phys.\ J. B, {\bf 29}, 167--171 (2002)
 
 \bibitem{FM}
Fukumoto,~Y., Moffatt,~T.~K.:  
Motion and expansion of a viscous vortex ring.
 Part 1. A higher-order asymptotic formula for the velocity.  
J.\ Fluid.\ Mech. {\bf 417}, 1--45 (2000)

 
 \bibitem{HHW2006}
Hao,~C., Hsiao,~L., Wang,~B.: 
Well-posedness for the fourth order nonlinear Schr\"odinger equations.  
J.\ Math.\ Anal.\ Appl. 
{\bf 320}, 246--265 (2006) 

\bibitem{HHW2007}
Hao,~C., Hsiao,~L., Wang,~B.:  
Well-posedness of Cauchy problem for the fourth order nonlinear Schr\"odinger equations in multi-dimensional spaces.  
J.\ Math.\ Anal.\ Appl.   
{\bf 328}, 58--83 (2007)

 
 \bibitem{HIT} 
Hirayama,~H., Ikeda,~M., Tanaka,~T.: 
 Well-posedness for the fourth-order Schr\"odinger equation 
 with third order derivative nonlinearities.   
NoDEA Nonlinear Differential Equations Appl.  
 {\bf 28}  Paper No. 46, 72 pp (2021)
 
 \bibitem{HJ2005}
Huo,~Z., Jia,~Y.:  
The Cauchy problem for the fourth-order nonlinear Schr\"odinger equation 
 related to the vortex filament.   
J.\ Differential Equations 
 {\bf  214}, 1--35 (2005)
 
 \bibitem{HJ2007}
 Huo,~Z., Jia,~Y.:  
 A refined well-posedness for the fourth-order nonlinear Schr\"odinger equation 
 related to the vortex filament.  
Comm.\ Partial Differential Equations 
{\bf 32}, 1493--1510 (2007)

\bibitem{HJ2011}
Huo,~Z., Jia,~Y.:  
Well-posedness for the fourth-order nonlinear derivative Schr\"odinger 
equation in higher dimension.   
J.\ Math.\ Pures Appl.   
{\bf 96}, 190--206 (2011) 

\bibitem{II}
Iorio,~R.,  Iorio,~V.~M.:  
Fourier Analysis and Partial Differential Equations.  
Cambridge Stud. Adv. Math. 70, Cambridge University Press (2001) 

\bibitem{Koiso1997}
Koiso,~N.:  
The vortex filament equation and a semilinear Schr\"odinger equation 
in a Hermitian symmetric space.  
Osaka J. Math. {\bf 34}, 199--214 (1997)

\bibitem{Kwon}
Kwon,~S.: 
On the fifth-order KdV equation: Local well-posedness and lack of uniform continuity of the solution map.  
J. Differential Equations  \textbf{245}, 2627--2659 (2008)   

\bibitem{LPD}
Lakshmanan,~M., Porsezian,~K., Daniel,~M.:   
Effect of discreteness on the continuum limit of the Heisenberg
spin chain.  
Phys.\ Lett.\ A \textbf{133}, 483--488 (1988) 

\bibitem{Mietka}
Mietka,~C.: 
On the well-posedness of a quasi-linear Korteweg-de Vries equation, 
Ann. Math. Blaise Pascal.  
\textbf{24}, 83--114 (2017) 

\bibitem{Mizuhara}
Mizuhara,~R.: 
The initial value problem for third and fourth order dispersive 
equations in one space dimension.  
Funkcial. Ekvac.  
\textbf{49}, 1--38 (2006) 
 
 \bibitem{NSVZ}
 Nahmod,~A., Shatah,~J., Vega,~L., Zeng,~C.: 
 Schr\"odinger maps and their associated frame systems.  
 Int.\ Math.\ Res.\ Not.\ IMRN. no. 21, Art. ID rnm088, 29 pp (2007)
 
\bibitem{onodera0} 
Onodera,~E.:  
Generalized Hasimoto transform of one-dimensional 
dispersive flows into compact Riemann surfaces.  
SIGMA Symmetry Integrability Geom.\ Methods Appl.  
\textbf{4},   article No. 044, 10 pp  (2008)

\bibitem{onodera4}
Onodera,~E.:  
Local existence of a fourth-order dispersive 
curve flow on locally Hermitian symmetric spaces and its application.  
Differential\ Geom.\ Appl.  
\textbf{67},  101560, 26 pp (2019)

\bibitem{onodera5}
Onodera,~E.: 
Uniqueness of 1D Generalized Bi-Schr\"odinger Flow. 
J.~Geom.~Anal.
\textbf{32}, Article number: 47, 41pp (2022)

\bibitem{onoderamomo}
Onodera,~E.:  
Structure of a fourth-order dispersive flow equation 
through the generalized Hasimoto transformation.  
J.~Geom.~Anal.
\textbf{34}, Article number: 347, 55pp (2024)

\bibitem{onodera6}
Onodera,~E.: 
Local well-posedness of the initial value problem for a fourth-order 
nonlinear dispersive system on the real line. 
preprint,  math.AP/2407.18605 

\bibitem{PDL} 
Porsezian,~k., Daniel,~M., Lakshmanan,~M.:   
On the integrability aspects of the one-dimensional classical 
continuum isotropic biquadratic Heisenberg spin chain.  
J. Math.\ Phys.  \textbf{33}, 1--10 (1992) 

\bibitem{RRS}
Rodnianski,~I, Rubinstein,~Y.,~A., Staffilani,~G.:   
On the global well-posedness of the one-dimensional 
Schr\"odinger map flow.   
Analysis and PDE \textbf{2}, 187--209 (2009) 

\bibitem{RWZ} 
Ruzhansky,~M., Wang,~B., Zhang,~H.: 
Global well-posedness and scattering for the fourth order nonlinear 
Schr\"odinger equations with small data in modulation and Sobolev spaces.   
J.\ Math.\ Pures Appl.   
{\bf 105}, 31--65 (2016) 

\bibitem{segata2003}
Segata,~J.: 
Well-posedness for the fourth-order nonlinear Schr\"odinger-type 
equation related to the vortex filament.   
Differential Integral Equations 
{\bf 16}, 841--864 (2003)

\bibitem{segata2004}
Segata,~J.:  
Remark on well-posedness for the fourth order nonlinear Schr\"odinger 
type equation. 
Proc.\ Amer.\ Math.\ Soc.  
{\bf 132}, 3559--3568 (2004)  
 
 \bibitem{segata}
 Segata,~J.:
 Refined energy inequality with application to well-posedness 
 for the fourth order nonlinear Schr\"odinger type equation 
 on torus. 
 J. Differential Equations  \textbf{252}, 5994--6011 (2012)   
 
 \bibitem{SW2011}
 Sun,~X.~W., Wang,~Y.~D.:
 KdV geometric flows on K\"ahler manifolds.  
 Internat.\ J.\ Math. \textbf{22}, 1439--1500 (2011) 
 
 \bibitem{SW2013}
 Sun,~X.~W., Wang,~Y.~D.:   
 Geometric Schr\"odinger-Airy Flows on K\"ahler Manifolds.  
 Acta Math.\ Sin. (Engl. Ser.) \textbf{29}, 209--240 (2013)
 
 \bibitem{WZY}
 Weng,~W., Zhang,~G., Yan,~Z.:  
 Strong and weak interactions of rational vector rogue waves and solitons 
 to any $n$-component nonlinear Schr\"odinger system 
 with higher-order effects.   
 Proc.~A. 
 \textbf{478} no. 2257, Paper No. 20210670, 23 pp (2022) 
 

\end{thebibliography}
\end{document}